\documentclass[Threeside,10pt]{article}
\usepackage [latin1]{inputenc}
\usepackage[english]{babel}
\usepackage{amsmath}
\usepackage{amsthm}
\usepackage{algorithm}
\usepackage{cite}
\usepackage{amsfonts}
\usepackage{amssymb}
\usepackage{mathrsfs}
\usepackage{graphicx}
\usepackage{colortbl,dcolumn}
\usepackage{marvosym}
\usepackage{ifsym}
\usepackage{paralist}
\usepackage{xcolor}
\usepackage[
colorlinks,
citecolor=blue,
linkcolor=blue,
backref=page
allcolors=black
]{hyperref} 
\allowdisplaybreaks

       \newtheorem{lemma}{\bf Lemma}[section]
       \newtheorem{theorem}{\bf Theorem}[section]
       \newtheorem{proposition}{\bf Proposition}[section]
       \newtheorem{corollary}{\bf Corollary}[section]
       \newtheorem{definition}{\bf Definition}[section]
       \newtheorem{remark}{\bf Remark}[section]

       \numberwithin{equation}{section}

\begin{document}
\title{{\sl Towards continuity: Universal frequency-preserving KAM persistence and remaining regularity}}
\author{Zhicheng Tong $^{a,1}$, Yong Li $^{*,b,1,2}$ }

\renewcommand{\thefootnote}{}
\footnotetext{\hspace*{-6mm}

\begin{tabular}{l l}
 $^{*}$~~~The corresponding author.\\
 $^{a}$~~~E-mail address : tongzc20@mails.jlu.edu.cn\\
 $^{b}$~~~E-mail address : liyong@jlu.edu.cn\\
 $^{1}$~~~School of Mathematics, Jilin University, Changchun 130012, People's Republic of China.\\
$^{2}$~~~School of Mathematics and Statistics, Center for Mathematics and Interdisciplinary Sciences,\\
~~~ Northeast Normal University, Changchun 130024, People's Republic of China.
\end{tabular}}

\date{}
\maketitle

\begin{abstract}
Beyond H\"{o}lder's type, this paper mainly concerns the persistence and remaining regularity of an individual frequency-preserving KAM torus in a finitely differentiable Hamiltonian system, even allows the  non-integrable part being critical  finitely smooth. To achieve this goal, besides investigating the Jackson approximation theorem towards only modulus of continuity, we demonstrate an abstract regularity theorem adapting to the new iterative scheme. Via these tools, we obtain a KAM theorem with sharp differentiability hypotheses, asserting that the persistent torus keeps prescribed universal Diophantine frequency unchanged. Further, the non-H\"older regularity for invariant KAM torus as well as the conjugation is explicitly shown by introducing asymptotic analysis. To our knowledge, this is the first approach to KAM on these aspects in a continuous sense, and we also provide two systems, which cannot be studied by previous KAM but by ours.
\end{abstract}

\tableofcontents

\section{Introduction}
 The celebrated KAM theory, due to  Kolmogorov and Arnold \cite{MR0068687,R-9,R-10,R-11}, Moser \cite{R-12,R-13}, mainly concerns the preservation of invariant tori of a Hamiltonian function $ H(y) $  under small perturbations (i.e., $H(y) \to H\left( {x,y,\varepsilon } \right) $ of freedom $ n \in \mathbb{N}^+ $ with $ \varepsilon>0 $ sufficiently small), and indeed has a history of more than sixty years. So far, KAM theory has been well developed and widely applied to a variety of  dynamical systems. For some other fundamental developments,  see Kuksin \cite{MR0911772},  Eliasson \cite{MR1001032}, P\"oschel \cite{MR1022821}, Wayne \cite{MR1040892}, Bourgain \cite{MR1316975,MR1345016} and so on.

 As is known to all, for frequency $ \omega  = {H_y}\left( {y} \right) $ of the unperturbed Hamiltonian  system, one often requires it to satisfy the following classical Diophantine condition (or be of Diophantine class $ \tau $)
 \begin{equation}\label{dio}
 	| {\langle {{\tilde k},\omega } \rangle } | \geq \alpha_ *{ | {{\tilde k }} |}^{-\tau} ,\;\;\forall 0 \ne \tilde k \in {\mathbb{Z}^n}
 \end{equation}
 with respect to Diophantine index $ \tau \geq n-1 $ and some $ \alpha_ *>0 $, where $ |\tilde k|: = \sum\nolimits_{j = 1}^n {|{{\tilde k}_j}|}  $. Otherwise, under a Liouville frequency that is not Diophantine, the torus may break no matter how small the perturbation is,  and some chaotic behavior can take place simultaneously. The strong (non-universal) Diophantine frequencies of class $ \tau=n-1 $,  are shown to be a continuum but only form a set of zero Lebesgue measure,  therefore the KAM preservation based on that is usually said to be \textit{non-universal}. As a contrast, the class $ \tau>n-1 $  Diophantine nonresonance KAM persistence is \textit{universal} because such frequencies are of full Lebesgue measure in $ \mathbb{R}^n $, and we will focus on this case throughout  this paper.
 
 Considering Diophantine nonresonance, the invariant tori are shown in classical KAM theorems to be preserved under analytic settings. It should be pointed out that, if a constant Diophantine frequency is prescribed in advance, one could obtain  KAM persistence with frequency being unchanged by proposing certain nondegeneracy or transversality conditions, which preserves more dynamics from the perturbed one, see  Salamon \cite{salamon}, Du et al  \cite{R-18} and Tong et al \cite{TD} for instance, respectively, and this analytic torus is interestingly shown to be never isolated by Eliasson et al in \cite{MR3357183}. Therefore,  beyond analyticity, it is still interesting to touch the minimum initial regularity for Hamiltonian to which KAM could applied. On this aspect, the first breakthrough was made by Moser in \cite{MR0147741}, where he studied twist maps (could correspond to an iso-energetic KAM theorem for $ n=2 $) admitting  finite smoothness of  integral order sufficiently large.  Much effort has been devoted on this technical problem in terms of H\"older continuity, including constructing counterexamples and reducing the differentiability hypotheses. For some classic fundamental work, see  Moser \cite{J-3}, Jacobowitz \cite{R-2}, Zehnder \cite{R-7,R-8}, Mather \cite{R-55},  Herman \cite{M1,M2}, P\"oschel \cite{Po1,Po2} and etc.  It is worth mentioning that, very recently  P\"oschel  announced in his preprint \cite{Po3} a  KAM theorem on the $ n $-dimensional torus $\mathbb{T}^n$ (without action variables) based on a non-universal  frequency  being of Diophantine  class $ \tau=n-1 $ in \eqref{dio}.  Specially, he pointed out that the  derivatives of order $ n $ need not be continuous, but rather $ L^2 $ in a certain strong sense, by introducing a new norm over $\mathbb{T}^n$. 
 
 Back to our concern on general $ n $-freedom Hamiltonian systems having action-angular variables, it is always conjectured that the minimum regularity requirement for the Hamiltonian function $ H(x,y) $ is at least $ C^{2n} $ due to the $C^{2n-\epsilon}$ (with any $\epsilon$ close to $0^+$) counterexample constructed by Cheng and Wang \cite{MR3061774}, which allows for all frequencies in $\mathbb{R}^n$, see also Wang \cite{MR4385768}.   However, via Diophantine nonresonance,  a more precise counterexample does not seem to have appeared. As to reducing initial regularity, along with the idea of Moser \cite{J-3}, the best known H\"older case $ C^{\ell} $ with  $ \ell >2\tau+2 $ has been established by Salamon in \cite{salamon},  the prescribed universal frequency under consideration is of Diophantine class $ \tau>n-1 $ in \eqref{dio}, and the remaining regularity of the frequency-preserving KAM torus as well as the conjugation is also showed to be H\"older's type. More precisely, the KAM torus is at least of class $ C^1 $, and the  conjugation from the dynamic on it to the linear flow  is at least of class  $ C^{\tau+1} $. Besides, the $C^{\ell}$ ($\ell>2\tau+2>2n$) differentiability hypotheses is sharp due to the counterexamples  of Herman  \cite{M1,M2} and Cheng and Wang \cite{MR3061774}. 
  In the aspect of H\"older's type, see Khesin et al \cite{MR3269186},  Bounemoura \cite{Bounemoura} and Koudjinan \cite{Koudjinan} for some other new developments. To strictly weaker than H\"older continuity,  a KAM theorem via non-universal Diophantine nonresonance of class $ \tau=n-1 $ in \eqref{dio} was proved by Albrecht in \cite{Chaotic},  claiming  that $ C^{2n} $  plus certain  modulus of continuity $ \varpi $ satisfying the classical Dini condition
 \begin{equation}\label{cdini}
 	\int_0^1 {\frac{{\varpi \left( x \right)}}{x}{\rm d}x}  <  + \infty
 \end{equation}
 is enough for the non-universal KAM persistence.   To the best of our knowledge,  there is no other work on KAM via only modulus of continuity except for \cite{Chaotic}.

Concerning about universal KAM persistence for finitely differentiable Hamiltonians of freedom $ n $, the best result so far, obtained by Salamon \cite{salamon}, still requires $ C^{2n} $ plus certain H\"older continuity depending on the Diophantine index. It is therefore natural that ones should consider the following fundamental questions successively, in order to further touch the criticality in this long standing finitely differentiable KAM  since Moser:
 
 \begin{itemize}
 	\item \textit{Can non-H\"{o}lder regularity allow for  KAM persistence with universal frequency-preserving?}

 	\item \textit{What kind of smoothness shall the invariant KAM torus admit, if it is indeed persistent? How about the conjugation?}
 	
 	\item \textit{Does there exist a Dini type integrability condition similar to \eqref{cdini} that reveals the explicit relation between nonresonance and regularity?}
 \end{itemize}
 
 
  Usually there are at least two ways reaching finitely smooth KAM:  abstract implicit function method began in Nash \cite{MR75639} and Zehnder \cite{R-7,R-8}, as well as the analytic smoothing approach due to Moser \cite{J-3}. We also refer to Hamilton \cite{MR0656198} and H\"ormander \cite{MR0802486} for some other versions, and we shall use the latter to touch the case with only continuity, because it provides an iterative scheme which could derive accurate remaining regularity for KAM torus and conjugation. As a consequence, towards continuity,  there are at least four difficulties to overcome.  Firstly,  note that the Jackson approximation theorem for classical H\"{o}lder continuity is no longer valid at present, hence it  must be developed to approximate the perturbed Hamiltonian function $  H\left( {x,y,\varepsilon } \right) $ in the sense of modulus of continuity,  as a crucial step. It needs to be emphasized that the analysis uses different approaches in comparison with the H\"older one, by proposing the semi separability on modulus of continuity.   Secondly, it is also basic how to establish a corresponding regularity iteration theorem to study the remaining regularity of the invariant torus and the conjugation  without H\"older's type. Thirdly, we have to set up a new KAM iterative scheme and prove its uniform convergence in the $C^1$ topology via these tools. Fourthly, it is somewhat difficult to extract an  integrability condition of universal nonresonance and initial regularity from KAM iteration, keeping the frequency-preserving KAM persistence. Indeed, to achieve  Main Theorem \ref{theorem1} having sharp differentiability hypotheses and \textit{frequency-preserving}, we apply Theorem  \ref{Theorem1}  to construct a series of analytic approximations to Hamiltonian $  H\left( {x,y,\varepsilon } \right) $ with modulus of continuity, and prove the persistence and  regularity of invariant torus via a modified  KAM iteration as well as a generalized Dini type condition. As some new efforts, our Main Theorem  \ref{theorem1} applies to a wide range in a continuous sense (even allows the non-integrable part to be non-H\"older), and reveals the sharp integral  relation between regularity and universal Diophantine nonresonance \textit{for the first time}. Apart from above, it is well known that small divisors must lead to certain loss of smoothness (such as reducing the  radius of analyticity), and only H\"older class has been investigated so far, see Salamon \cite{salamon}. If only continuity for highest order  derivatives is assumed, obviously ones should not expect the KAM remaining regularity still be of H\"older's type. \textit{On this aspect, our Main Theorem \ref{theorem1} gives the first approach to the  regularity of invariant KAM torus and conjugation being non-H\"older, explicitly shown by certain   asymptotic analysis.} Besides, as shown by Theorem \ref{lognew}, the KAM torus and conjugation may  interestingly admit a completely different class of regularity, due to the affect of the Diophantine index. Particularly, as a direct  application, our Main Theorem \ref{theorem1} could deal with  the case of general modulus of continuity for $ H\left( {x,y,\varepsilon }\right) $, such as Logarithmic H\"{o}lder continuity case, that is, for all $ 0 < \left| {x - \xi } \right| + \left| {y - \eta } \right| \leq 1/2 $,
 \begin{displaymath}
 	\left| {{\partial ^\alpha }H\left( {x,y,\varepsilon } \right) - {\partial ^\alpha }H\left( {\xi ,\eta ,\varepsilon } \right)} \right| \leq \frac{c}{{{{\left( { - \ln \left( {\left| {x - \xi } \right| + \left| {y - \eta } \right|} \right)} \right)}^\lambda }}}
 \end{displaymath}
 with respect to all $ \alpha  \in {\mathbb{N}^{2n}}$ with $ \left| \alpha  \right| = {2n } $, where $ n \geq 2 $,   $\lambda>1$, $n-1<\tau \in \mathbb{N}^+$, $c, \varepsilon>0 $ are sufficiently small,  $ \left( {x,y} \right) \in {\mathbb{T}^n} \times G $ with $ {\mathbb{T}^n}: = {\mathbb{R}^n}/ \mathbb{Z}^n $, and $ G \subset {\mathbb{R}^n} $ is a connected closed set with interior points. One shall notice that, this Hamiltonian system with non-H\"older continuity  cannot be studied by any available KAM theorems so far, see Section \ref{section6} for more details.
 
 This paper is organized as follows. In  Section \ref{section2}, by introducing some basic notions and properties for modulus of continuity, we establish a Jackson type approximation theorem in a continuous sense, and the proof will be postponed to Section \ref{JACK}. Then we state our main result in this paper. Namely, considering that the highest order derivatives of Hamiltonian function $ H $ with respect to the action-angular variables $(x,y)$ are only continuous, we present a KAM theorem (Main Theorem \ref{theorem1}) with sharp differentiability hypotheses under certain assumptions, involving a generalized Dini type integrability condition (H1). The applications of this result are given in Section \ref{section6}, including H\"{o}lder, H\"{o}lder plus Logarithmic H\"{o}lder and a more complicated circumstance, aiming to show the importance and universality of Main  Theorem \ref{theorem1}. In particular,  two explicit Hamiltonian  systems are constructed, which cannot be studied by any KAM theorems for finite smoothness via classical H\"{o}lder continuity, namely having a  \textit{non-H\"older corner} and even being of \textit{nowhere H\"older's type}, while our Main Theorem \ref{theorem1} still works.  Section \ref{section4} provides the proof of Main Theorem \ref{theorem1} and is mainly divided into two parts: the first part deals with the modified KAM steps via only modulus of continuity, while the second part is devoted to giving an iteration  theorem (Theorem \ref{t1}) on regularity,  which is essential in analyzing the KAM remaining smoothness for the invariant torus as well as the conjugation. Finally Section \ref{ProofOther} presents the proof of Theorems  \ref{Holder},  \ref{lognew} and \ref{GLHnew} in Section \ref{section6}, respectively.

\section{Statement of results}
\label{section2}

\subsection{Preliminaries for modulus of continuity}
We first give some notions, including the modulus of continuity along with the norm based on it,  the  semi separability which will be used in  Theorem \ref{Theorem1}, as well as the weak homogeneity which will appear in Theorem \ref{theorem1}.

Denote by $ |\cdot| $ the sup-norm in $ \mathbb{R}^d $ and the dimension $ d \in \mathbb{N}^+ $ may vary throughout this paper. We formulate that in the limit process, $ f_1(x)=\mathcal{O}^{\#}\left(f_2(x)\right) $ means there are absolute positive constants $ \ell_1 $ and $ \ell_2 $ such that $ {\ell _1}{f_2}\left( x \right) \leq {f_1}\left( x \right) \leq {\ell _2}{f_2}\left( x \right) $, and $ f_1(x)=\mathcal{O}\left(f_2(x)\right)  $ implies that there exists an absolute positive constant $ \ell_3 $ such that $ |f_1(x)| \leq \ell_3 f_2(x) $, and finally $ f_1(x)\sim f_2(x)  $ indicates that $ f_1(x) $ and $ f_2(x) $ are equivalent.

\begin{definition}[Modulus of continuity]\label{d1}
	A nondecreasing continuous function $ \varpi (t)>0 $ on the interval $ \left( {0,\delta } \right] $ with respect to some $ \delta  >0 $ is said to be a modulus of continuity, if $ \mathop {\lim }\limits_{x \to {0^ + }} \varpi \left( x \right) = 0 $ and $ \mathop {\overline {\lim } }\limits_{x \to {0^ + }} x/{\varpi }\left( x \right) <  + \infty  $.  Next, we define the following semi norm and norm for a continuous function $ f $ on $ {\mathbb{R}^n} $: 
	\begin{equation}\notag
		{\left[ f \right]_\varpi }: = \mathop {\sup }\limits_{x,y \in {\mathbb{R}^n},\;0 < \left| {x - y} \right| \leq \delta } \frac{{\left| {f\left( x \right) - f\left( y \right)} \right|}}{{\varpi \left( {\left| {x - y} \right|} \right)}},\;\;{\left| f \right|_{{C^0}}}: = \mathop {\sup }\limits_{x \in {\mathbb{R}^n}} \left| {f\left( x \right)} \right|.
	\end{equation}
	We call  $ f $ $ C_{k,\varpi}  $ continuous if $ f $ has partial derivatives $ {{\partial ^\alpha }f} $ for $ \alpha=(\alpha_1, \ldots ,\alpha_n)\in \mathbb{N}^n,  \left| \alpha  \right|:=\sum\nolimits_{i = 1}^n {\left| {{\alpha _i}} \right|}  \leq k \in \mathbb{N} $ and satisfies
	\begin{equation}\label{k-w}
		{\left\| f \right\|_\varpi }: = \sum\limits_{\left| \alpha  \right| \leq k} {\left( {{{\left| {{\partial ^\alpha }f} \right|}_{{C^0}}} + {{\left[ {{\partial ^\alpha }f} \right]}_\varpi }} \right)}  <  + \infty .
	\end{equation}
	Denote by $ {C_{k,\varpi }}\left( {{\mathbb{R}^n}} \right) $ the space composed of all functions $ f $ satisfying \eqref{k-w}.
\end{definition}
\begin{remark}
	For $ f:{\mathbb{R}^n} \to \Omega  \subset {\mathbb{R}^d} $ with a modulus of continuity $ \varpi $, we modify the above designation to $ {C_{k,\varpi }}\left( {{\mathbb{R}^n},\Omega } \right) $.
\end{remark}

 It can be seen that the well-known Lipschitz continuity and H\"{o}lder continuity are special cases in the above definition. In particular, for $ 0<\ell \notin \mathbb{N}^+ $, we denote by $ f \in {C^\ell }\left( {{\mathbb{R}^n}} \right) $ the function space in which the higher derivatives in $ \mathbb{R}^n $ are $ \{\ell\} $-H\"{o}lder continuous, that is, the modulus of continuity is of the form 
\[\varpi_{\mathrm{H}}^{\{\ell\}}(x)\sim x^{\ell},\;\; x \to 0^+,\]
 where $ \{\ell\} \in (0,1)$ denotes the fractional part of $ \ell $. As a generalization of classical H\"{o}lder continuity, we define the \textit{Logarithmic H\"{o}lder continuity} with index $ \lambda  > 0 $, where 
\[\varpi_{\mathrm{LH}}^{\lambda} \left( x \right) \sim 1/{\left( { - \ln x} \right)^\lambda },\;\;x \to 0^+,\]
 and we will omit the the range $ 0 < x   \ll 1 $ without causing ambiguity for these kind of functions with singularities away from $ 0^+ $, because one only needs to focus on the asymptotic behavior of a given modulus of continuity near $ 0^+ $. Moreover, one could further consider \textit{the  generalized Logarithmic H\"older's type}  with indices $ \varrho\in \mathbb{N}^+ $ and $ \lambda>0 $ as follows:
\begin{equation}\label{mafan}
	\varpi_{\mathrm{GLH}}^{\varrho,\lambda} \left( x \right) \sim \frac{1}{{(\ln (1/x))(\ln \ln (1/x)) \cdots {{(\underbrace {\ln  \cdots \ln }_\varrho (1/x))}^\lambda }}},\;\;x \to 0^+,
\end{equation}
and we have $ 	\varpi_{\mathrm{GLH}}^{1,\lambda} \left( x \right) \sim \varpi_{\mathrm{LH}}^{\lambda} \left( x \right) $.

\begin{remark}\label{rema666}
	It is well known that a mapping defined on a bounded connected closed set in a finite dimensional space must have a modulus of continuity, see \cite{Herman3,MR1036903}. For example, for a function $ f(x) $ defined on $ [0,1]  \subset {\mathbb{R}^1} $, it automatically admits a modulus of continuity
	\[{\omega _{f,\delta }}\left( x \right): = \mathop {\sup }\limits_{y \in \left[ {0,1} \right],0 < \left| {x - y} \right| \leq \delta } \left| {f\left( x \right) - f\left( y \right)} \right|.\]
\end{remark}
In view of Remark \ref{rema666}, we therefore only focus on modulus of continuity throughout this paper. Next we introduce the comparison relation between the strength and the weakness of modulus of continuity.

\begin{definition}\label{d5}
	Let $ {\varpi _1} $ and $ {\varpi _2} $ be modulus of continuity on interval $ \left( {0,\delta } \right] $. We say that  $ {\varpi _1} $ is  weaker (strictly weaker) than $ {\varpi _2} $ if $ \mathop {\overline\lim }\limits_{x \to {0^ + }} {\varpi _2}\left( x \right)/{\varpi _1}\left( x \right) <+\infty $ ($ =0 $).
\end{definition}
\begin{remark}\label{strict}
	Obviously any modulus of continuity is weaker than Lipschitz's type, and  the Logarithmic H\"{o}lder's type $ \varpi_{\mathrm{LH}}^{\lambda} \left( x \right) \sim 1/{\left( { - \ln x} \right)^\lambda } $ with any $ \lambda  > 0 $ is strictly weaker than \textit{arbitrary}  H\"{o}lder's type $ \varpi_{\mathrm{H}}^{\alpha}\left( x \right) \sim {x^\alpha } $ with any $ 0 < \alpha  < 1 $. The  generalized Logarithmic H\"older's type $ \varpi_{\mathrm{GLH}}^{\varrho,\lambda} \left( x \right) $ in \eqref{mafan} with $ \varrho\in \mathbb{N}^+ $ and $ \lambda>0 $ is weaker than both of them.
\end{remark}

As a consequence, one directly obtains the  following corollary by Remark \ref{strict}. This shows that it is indeed very necessary for us to extend H\"older's continuous KAM to only continuous type.

\begin{corollary}\label{V37-Re2.3}
Give a bounded connected closed set $ \Omega  \subset \mathbb{R}^n $ with interior points. Then for every $ k \in \mathbb{N} $, we have $ {C_{k,\varpi_{\mathrm{H}}^{\alpha} }}\left( \Omega  \right)  \subseteq  {C_{k,\varpi_{\mathrm{LH}}^{\lambda} }} \left( \Omega  \right)  \subset {C_{k,\varpi_{\mathrm{GLH}}^{\varrho,\lambda}  }}\left( \Omega  \right) $ with arbitrary given $ 0<\alpha<1, \lambda>0 $ and $ \varrho\in \mathbb{N}^+ $. 
\end{corollary}

Finally, in order to emphasize the rigor of the KAM analysis throughout  this paper, we have to introduce two crucial concepts for modulus of continuity, namely the \textit{semi separability} and the \textit{weak homogeneity}.

\begin{definition}[Semi separability]\label{d2}
	A modulus of continuity $ \varpi  $ is said to be  semi separable, if for $ x \geq 1 $, there holds
	\begin{equation}\label{Ox}
		\psi \left( x \right): = \mathop {\sup }\limits_{0 < r < \delta /x} \frac{{\varpi \left( {rx} \right)}}{{\varpi \left( r \right)}} = \mathcal{O}\left( x \right),\;\;x \to  + \infty .
	\end{equation}
\end{definition}
\begin{remark}\label{Remarksemi}
	Semi separability directly leads to $ \varpi \left( {rx} \right) \leq \varpi \left( r \right)\psi \left( x \right) $ for $ 0 < rx \leq \delta  $, which will be used in the proof of the  Jackson type  Theorem \ref{Theorem1} via only modulus of continuity.
\end{remark}

\begin{definition}[Weak homogeneity] \label{weak}
	A modulus of continuity $ \varpi $ is said to admit weak homogeneity, if for fixed $ 0<a<1 $, there holds
	\begin{equation}\label{erfenzhiyi}
		\mathop {\overline {\lim } }\limits_{x \to {0^ + }} \frac{{\varpi \left( x \right)}}{{\varpi \left( {ax} \right)}} <  + \infty .
	\end{equation}
\end{definition}
\begin{remark}
The weak homogeneity plays  a controlling role in KAM iteration, see details from the proof of Main Theorem \ref{theorem1}.
\end{remark}

It should be emphasized that semi separability and weak homogeneity are universal hypotheses.	The H\"older and Lipschitz types  automatically admit them. Many modulus of continuity weaker than the H\"older one are semi separable and also admit weak homogeneity, e.g., for the Logarithmic H\"{o}lder's type $ \varpi_{\mathrm{LH}}^{\lambda} \left( x \right) \sim 1/{\left( { - \ln x} \right)^\lambda } $ with any  $ \lambda  > 0 $, one verifies that $ \psi \left( x \right) \sim {\left( {\ln x} \right)^\lambda } = \mathcal{O}\left( x \right) $ as $ x \to +\infty $ in \eqref{Ox}, and $ \mathop {\overline {\lim } }\limits_{x \to {0^ + }} {\varpi_{\mathrm{LH}}^{\lambda}}\left( x \right)/{\varpi_{\mathrm{LH}}^{\lambda}}\left( {ax} \right) = 1 <  + \infty  $ with all $ 0<a<1 $ in \eqref{erfenzhiyi}. See  examples in Lemmas \ref{Oxlemma} and \ref{ruotux},  in particular, \textit{it is pointed out that a convex modulus of continuity  naturally possesses these two properties.}

\subsection{Jackson type approximation theorem beyond H\"older class}

This section provides a Jackson type approximation theorem beyond H\"older's type and some related corollaries based on Definitions \ref{d1} and \ref{d2}. Their  proof will be postponed to  Sections  \ref{JACK}, \ref{proofcoro1} and \ref{proofcoco2}, respectively. We shall emphasize that, the H\"older's conclusion is relatively easy to obtain, because the H\"older continuity $ \varpi_{\mathrm{H}}^{\alpha}(x) $ with $ 0<\alpha<1 $ is naturally homogeneous, i.e.,  $ \varpi_{\mathrm{H}}^{\alpha}(bx)=b^{\alpha} \varpi_{\mathrm{H}}^{\alpha}(x) $ holds for any $ b>0 $, and it is indeed crucial in the analysis. As to the general modulus of continuity, one has to propose semi separability  instead of this (in a weaker sense)  and the analysis becomes more complicated, namely  estimating the integrals separately, under different approaches, see details from Section \ref{JACK}.

\begin{theorem}\label{Theorem1}
	There is a family of convolution operators
	\begin{equation}\notag
		{S_r}f\left( x \right) = {r^{ - n}}\int_{{\mathbb{R}^n}} {K\left( {{r^{ - 1}}\left( {x - y} \right)} \right)f\left( y \right){\rm d}y} ,\;\;0 < r \leq 1
	\end{equation}
	from $ {C^0}\left( {{\mathbb{R}^n}} \right) $ into the space of entire functions on $ {\mathbb{C}^n} $ with the following property. For every $ k \in \mathbb{N} $, there exists a constant $ c\left( {n,k} \right)>0 $ such that, for every $ f \in {C_{k,\varpi }}\left( {{\mathbb{R}^n}} \right) $ with a semi separable modulus of continuity $ \varpi $, every multi-index $ \alpha  \in {\mathbb{N}^n} $ with $ \left| \alpha  \right| \leq k  $, and every $ x \in {\mathbb{C}^n} $ with $ \left| {\operatorname{Im} x} \right| \leq r $, we have
	\begin{equation}\label{3.2}
		\left| {{\partial ^\alpha }{S_r}f\left( x \right) - {P_{{\partial ^\alpha }f,k - \left| \alpha  \right|}}\left( {\operatorname{Re} x;\mathrm{i}\operatorname{Im} x} \right)} \right| 		\leq c\left( {n,k} \right){\left\| f \right\|_\varpi }{r^{k - \left| \alpha  \right|}}\varpi(r),
	\end{equation}
	where the Taylor polynomial $ P $ is defined as follows
	\[{P_{f,k}}\left( {x;y} \right) := \sum\limits_{\left| \beta  \right| \leq k} {\frac{1}{{\alpha !}}{\partial ^\beta }f\left( x \right){y^\alpha }}. \]
	Moreover,  $ {{S_{r}}f} $ is real analytic whenever $ f $ is real valued.
\end{theorem}

As a direct consequence of Theorem \ref{Theorem1}, we give the following Corollaries \ref{coro1} and \ref{coro2}. These results have been  widely used in H\"older's case, see for instance, \cite{Koudjinan,salamon,MR2071231}.

\begin{corollary}\label{coro1}
	The approximation function $ {{S_r}f\left( x \right)} $ in  Theorem \ref{Theorem1} satisfies
	\begin{equation}\notag
		\left| {{\partial ^\alpha }\left( {{S_r}f\left( x \right) - f\left( x \right)} \right)} \right| \leq c_*{\left\| f \right\|_\varpi }{r^{k - \left| \alpha  \right|}}\varpi(r)
	\end{equation}
	and
	\begin{equation}\notag
		\left| {{\partial ^\alpha }{S_r}f\left( x \right)} \right| \leq {c^ * }{\left\| f \right\|_\varpi }
	\end{equation}
	for  $ x \in \mathbb{C}^n $ with $ \left| {\operatorname{Im} x} \right| \leq r $, $ |\alpha| \leq k $, where $ c_* = c_*\left( {n,k} \right) >0$ and $ {c^ * } = {c^ * }\left( {n,k,\varpi } \right) >0$ are some universal constants.
\end{corollary}

\begin{corollary}\label{coro2}
	If the function $ f\left( x \right) $ in  Theorem \ref{Theorem1} also satisfies that the period of each variables $ {x_1}, \ldots ,{x_n} $ is $ 1 $ and the integral over $ {\mathbb{T}^n} $ is zero, then the approximation function $ {S_r}f\left( x \right) $ also satisfies these properties.
\end{corollary}

\subsection{Universal frequency-preserving KAM via modulus of continuity and remaining regularity}
We are now in a position to state the universal frequency-preserving KAM theorem via only modulus of continuity in this paper. Before this, let's start with our parameter settings.
Let $ n \geq 2$ (degree of freedom),  $\tau  > n - 1$ (Diophantine index), $  2\tau  + 2 \leq k \in {\mathbb{N}^ + } $ (differentiable order) and a sufficiently large number $ M>0 $ be given, throughout all KAM theorems in this paper. Consider a Hamiltonian function with action-angular variables $ H(x,y):{\mathbb{T}^n} \times G \to \mathbb{R} $ with $ {\mathbb{T}^n}: = {\mathbb{R}^n}/ \mathbb{Z}^n $, and $ G \subset {\mathbb{R}^n} $ is a connected closed set with interior points. It follows from Remark \ref{rema666} that $ H(x,y) $ automatically admits a modulus of continuity $ \varpi $. In view of the comments below Definition \ref{weak}, we assume that $ \varpi $ admits semi separability  and weak homogeneity without loss of generality.  Besides, we make the following assumptions:
\begin{itemize}
	\item[(H1)] Dini type integrability condition: Assume that $ H(x,y)\in {C_{k,\varpi }}\left( {{\mathbb{T}^n} \times G} \right) $ with the above modulus of continuity $ \varpi $. In other words, $ H(x,y) $ has at least  derivatives of order $ k $, and the highest order  derivatives admit the regularity of $ \varpi $. Moreover, $ \varpi $ satisfies the Dini type  integrability condition
	\begin{equation}\label{Dini}
		\int_0^1 {\frac{{\varpi \left( x \right)}}{{{x^{2\tau  + 3 - k}}}}{\rm d}x}  <  + \infty .
	\end{equation}
	
	\item[(H2)] Boundedness and nondegeneracy:
	\begin{equation}\notag
		{\left\| H \right\|_\varpi } \leq M,\;\;\left| {{{\left( {\int_{{\mathbb{T}^n}} {{H_{yy}}\left( {\xi ,0} \right){\rm d}\xi } } \right)}^{ - 1}}} \right| \leq M.
	\end{equation}
	
	\item[(H3)] Diophantine condition: For some $ \alpha_ *  > 0 $, the prescribed frequency $ \omega  \in {\mathbb{R}^n} $ satisfies
	\begin{equation}\notag
		| {\langle {{\tilde k},\omega } \rangle } | \geq \alpha_ *{ | {{\tilde k }} |}^{-\tau} ,\;\;\forall 0 \ne \tilde k \in {\mathbb{Z}^n},\;\; |\tilde k|: = \sum\limits_{j = 1}^n {|{{\tilde k}_j}|}   .
	\end{equation}

	\item[(H4)] KAM smallness: There holds
	\begin{align}\label{T1-2}
		&\sum\limits_{\left| \alpha  \right| \leq k} {\left| {{\partial ^\alpha }\Big( {H\left( {x,0} \right) - \int_{{\mathbb{T}^n}} {H\left( {\xi ,0} \right){\rm d}\xi } } \Big)} \right|{\varepsilon ^{\left| \alpha  \right|}}}  \notag \\
		+ &\sum\limits_{\left| \alpha  \right| \leq k - 1} {\left| {{\partial ^\alpha }\left( {{H_y}\left( {x,0} \right) - \omega } \right)} \right|{\varepsilon ^{\left| \alpha  \right| + \tau  + 1}}}  \leq M{\varepsilon ^k}\varpi \left( \varepsilon  \right)
	\end{align}
	for every $ x \in \mathbb{R}^n $ and some constant $ 0 < \varepsilon  \leq {\varepsilon ^ * } $.
	
	\item[(H5)]    Criticality: For $ \varphi_i(x):=x^{k-(3-i)\tau-1}\varpi(x) $ with $ i=1,2 $, there exist critical $ k_i^*\in \mathbb{N}^+ $ such that
	\[\int_0^1 {\frac{{{\varphi _i}\left( x \right)}}{{{x^{k_i^ *+1 }}}}{\rm d}x}  <  + \infty ,\;\;\int_0^1 {\frac{{{\varphi _i}\left( x \right)}}{{{x^{k_i^ *  + 2}}}}{\rm d}x}  =  + \infty .\]
\end{itemize}

Let us make some comments on our assumptions.
\begin{itemize}
	

	
		\item[(C1)] Our Dini type integrability condition (H1)  reveals the \textit{deep  relationship} between the \textit{irrationality} for frequency $ \omega $, \textit{order} and \textit{continuity} of  the highest order derivatives for the Hamiltonian $ H(x,y) $, and appears in universal KAM  for the first time.  It also admits sharpness, because $ H(x,y) $ could be $ C^{[2\tau+2]} $ differentiable, close to the counterexample constructed in \cite{MR3061774}.

	\item[(C2)]   According to the properties of Banach algebra, for the H\"older's type, it is assumed that (H4) only needs the term of $ \left| \alpha  \right| = 0 $, and does not need higher order derivatives to satisfy the condition. However, for general modulus of continuity, it seems not easy to establish the corresponding Banach algebraic properties, we thus add higher order derivatives in (H4). 
	Sometimes they can be removed correspondingly.
	
	\item[(C3)] The existence of $ k_i^* $ in (H5)  is directly guaranteed by (H1), actually this assumption is proposed to investigate the higher regularity of the persistent KAM torus as well as the conjugation, that is, the regularity  to $ C^{k_i^*} $ plus certain kinds of modulus of continuity. In general, given an explicit modulus of continuity $ \varpi $, such $ k_i^* $ in (H5) are automatically  determined by using asymptotic analysis, see details from Section \ref{section6}.
\end{itemize}

Finally, under sharp differentiability of $ C^{[2\tau+2]} $, we state the following Main Theorem \ref{theorem1} involving \textit{frequency-preserving KAM persistence} as well as \textit{remaining regualrity for invariant torus and conjugation.} \textit{To the best of our knowledge, this is the first approach  on these two aspects, towards KAM via only continuity.}

\begin{theorem}[Main Theorem]\label{theorem1}
(Part I)	Assume (H1)-(H4). Then there is a solution
	\[x = u\left( \xi  \right),\;\;y = v\left( \xi  \right)\]
	of the following equation with the operator $ D: = \sum\limits_{\nu  = 1}^n {{\omega _\nu }\frac{\partial }{{\partial {\xi _\nu }}}} = \omega  \cdot {\partial _\xi }  $
	\begin{equation}\notag
		Du = {H_y}\left( {u,v} \right),\;\;Dv =  - {H_x}\left( {u,v} \right),
	\end{equation}
	such that $ u\left( \xi  \right) - \xi  $ and $ v\left( \xi  \right) $ are of period $ 1 $ in all variables, where $ u $ and $ v $ are at least $ C^1 $. \vspace{3mm}
\\
(Part II)	In addition, assume (H5). Then there exist modulus of continuity $ {\varpi _i} $ ($ i=1,2 $) such that $ u \in {C_{k_1^ * ,{\varpi _1}}}\left( {{\mathbb{R}^n},{\mathbb{R}^n}} \right) $ and $ v \circ {u^{ - 1}} \in {C_{k_2^ * ,{\varpi _2}}}\left( {{\mathbb{R}^n},G} \right) $. Particularly, $ {\varpi _i} $ can be explicitly  determined as follows
	\begin{equation}\label{varpii}
		{\varpi _ i }\left( \gamma \right) \sim \gamma \int_{L_i\left( \gamma  \right)}^\varepsilon  {\frac{{\varphi_i \left( t \right)}}{{{t^{{k_i^*} + 2}}}}{\rm d}t}  = {\mathcal{O}^\# }\left( {\int_0^{L_i\left( \gamma  \right)} {\frac{{\varphi_i \left( t \right)}}{{{t^{{k_i^*} + 1}}}}{\rm d}t} } \right),\;\;\gamma  \to {0^ + },
	\end{equation}
	where $ L_i(\gamma) \to 0^+ $  are some functions such that the second relation in \eqref{varpii} holds for $ i=1,2 $.	
\end{theorem}

\begin{remark}\label{xiangxi}
	We say such a solution $x=u(\xi),y=v(\xi)$ the KAM one. Besides, the regularity for $ v \circ {u^{ - 1}} $ and $ u $ represents to the smoothness of  the invariant torus and the conjugation from the dynamic on it to the linear flow, respectively.
\end{remark}

\begin{remark}\label{RemarkM2}
	\begin{itemize}
	\item[(a)] 	With the same as in \cite{salamon}, the unperturbed system under consideration might be non-integrable (e.g., $ H = \left\langle {\omega ,y} \right\rangle  + \left\langle {A\left( x \right)y,y} \right\rangle  +  \cdots  $), which is different from the requirement of having  integrable part in \cite{Po2,Bounemoura}, and the KAM persistence is of frequency-preserving.
	
	\item[(b)] A Hamiltonian system having non-integrable unperturbed  part  is generally not reduced to a nearly integrable one, see Arnold et al \cite{MR2269239}.
	
	\item[(c)]  Moreover, our non-integrable part here could only be $ C_{k,\varpi} $ finitely differentiable with $ k \geq 2\tau+2 $ (we will see later that the critical case  $ k=[2\tau+2] $ could be indeed achieved in Theorems \ref{Holder}, \ref{lognew} and \ref{GLHnew}), while considering the  integrable part, \cite{Chaotic} and \cite{Po2} require the   real-analyticity, and \cite{Bounemoura} requires the H\"older continuity $ C^{\varsigma } $  with $ \varsigma>2\tau+4 $.
	\end{itemize}
\end{remark}
\begin{remark}
	Actually  Theorem \ref{theorem1} provides a method for determining $ \varpi_i $ with $ i=1,2 $, see \eqref{varpii}. For the prescribed modulus of continuity to  Hamiltonian, such as the H\"older and Logarithmic H\"older types, we have to use asymptotic analysis to  derive the remaining  continuity explicitly  in Section \ref{section6}, which is somewhat difficult.
\end{remark}

As mentioned forego, the H\"older's type $ H \in  C^{\ell}(\mathbb{T}^n,G) $ with $ \ell>2\tau+2 $ (where $ \tau>n-1$ is the Diophantine index) is always regarded as the critical case in the sense of H\"older due to \cite{salamon,Bounemoura,MR3061774}. Now let us consider the non-universal Diophantine nonresonance, i.e., $ \tau=n-1 $, and such frequencies can only form a set of zero Lebesgue measure. Then one observes that $ k=2\tau+2=2n $ is the critical degenerate case in our settings, and our Dini type integrability condition \eqref{Dini} in (H1) becomes the classical Dini condition \eqref{cdini}, which also allows  regularity weaker than arbitrary H\"older's type. This is exactly the same regularity required in \cite{Chaotic}, concentrating on non-universal KAM persistence without frequency-preserving. However, the non-universal Diophantine frequencies with index $ \tau=n-1 $ are not enough to represent almost all frequencies in $ \mathbb{R}^n $. 

 Back to our concern on universal KAM persistence, we may have to require the generalized Dini  condition in (H1) instead of \eqref{cdini}. Obviously, (H1) automatically holds if the highest differentiable order $ k $ of $ H(x,y) $ satisfies $ k\geq 2\tau+3 $ or even larger,  because the modulus of continuity $ \varpi $ does not have a singularity at $ 0 $, and therefore so does the integrand. But our Main Theorem \ref{theorem1} still makes sense, because the remaining regularity of the persistent torus as well as the conjuation will be higher correspondingly, as one would expect.

\section{Applications}
\label{section6}

In this section, we show certain detailed regularity  about KAM torus as well as conjugation such as H\"{o}lder and Logarithmic H\"{o}lder ones etc. Denote by $ \{a\} $ and $ [a] $ the fractional part and the integer part  of $ a\geq0 $, respectively. It should be emphasized that the Dini type  integrability condition \eqref{Dini} in (H1) is easy to verify, that is, the frequency-preserving KAM persistence is relatively easy to obtain. However, some complicated techniques of asymptotic analysis are needed to investigate the specific regularity of KAM torus as well as the conjugation, which are mainly reflected in the selection of functions $ L_i (\gamma)$ ($ i=1,2 $) in \eqref{varpii}. In particular, caused by small divisors, we will explicitly see how much regularity will be lost,  see, for instance, Theorems \ref{Holder}, \ref{lognew}, \ref{GLHnew} and the examples shown in Section \ref{explicitexa}. In what follows, the modulus of continuity under consideration are always convex near $ 0^+ $ and therefore automatically admit semi separability as well as weak homogeneity which we forego.


\subsection{KAM via different continuity}\label{subun}

Recall that we require $\tau>n-1$ and $n \geq 2$ in  (H3), that is, the prescribed Diophantine frequency is universal. Under such settings, the known minimum regularity requirement for Hamiltonian $ H(x,y) $ is H\"older's type $ C^{\ell} $ with $ \ell>2\tau +2 $, see Salamon's KAM in  \cite{salamon} as well as Theorem \ref{Holder} below. Interestingly, if one considers weaker modulus of continuity, such as $ C^{2\tau+2} $ plus  Logarithmic H\"older's type, the above regularity $C^\ell$ could be weakened, which can be seen from our new  Theorems \ref{lognew} and \ref{GLHnew}. It should be noted that weakening the initial regularity \textit{does not} mean that KAM results with higher initial regularity are fully included, because we also have to consider the remaining KAM regularity, namely the smoothness of the KAM torus and conjugation. Intuitively and obviously, the weaker the initial regularity, the weaker the remaining regularity. One should not expect the KAM torus and conjugation to be H\"older's type when considering Hamiltonian $ H(x,y) $ via only continuous high order derivatives.

\subsubsection{H\"{o}lder continuous case}\label{subsubsub}
Let us discuss the classical H\"{o}lder case first, see also the result given in \cite{salamon}. As we will see that, the remaining regularity is still of H\"{o}lder's type.
\begin{theorem}\label{Holder}
	Let $ H(x,y) \in C^{\ell} (\mathbb{T}^n,G) $ with $ \ell>2\tau +2 $, where  $ \ell  \notin \mathbb{N}^+ $, $ \ell-\tau \notin \mathbb{N}^+$ and $ \ell-2\tau \notin \mathbb{N}^+ $. That is, $ H(x,y) $ is of $ C_{k,\varpi} $ with $ k=[\ell] $ and $ \varpi(x)\sim \varpi_{\mathrm{H}}^{\ell}(x)\sim x^{\{\ell\}} $.	Assume (H2), (H3) and (H4). Then there is a solution $ x = u\left( \xi  \right),y = v\left( \xi  \right) $	of the following equation with the operator $ D: = \sum\limits_{\nu  = 1}^n {{\omega _\nu }\frac{\partial }{{\partial {\xi _\nu }}}}= \omega  \cdot {\partial _\xi }  $
	\begin{equation}\notag
		Du = {H_y}\left( {u,v} \right),\;\;Dv =  - {H_x}\left( {u,v} \right)
	\end{equation}
	such that $ u\left( \xi  \right) - \xi  $ and $ v\left( \xi  \right) $ are of period $ 1 $ in all variables. In addition, $ u \in {C^{\ell-2\tau-1}}\left( {{\mathbb{R}^n},{\mathbb{R}^n}} \right) $ and $ v \circ {u^{ - 1}} \in {C^{\ell-\tau-1}}\left( {{\mathbb{R}^n},G} \right) $.
\end{theorem}


\subsubsection{H\"{o}lder plus Logarithmic H\"{o}lder continuous case}\label{SUBLH}
To explicitly show different modulus of continuity strictly weaker than H\"older's type, we establish the following Theorem \ref{lognew}. One will see later that Theorem \ref{lognew} employs  very complicated asymptotic analysis, and interestingly, \textit{the remaining regularity basically characterized by $ \varpi_1 $ and $ \varpi_2 $ admit different forms.} 
\begin{theorem}\label{lognew}
	Let $ \tau>n-1 $ be given and let $ H(x,y)\in C_{[2\tau+2], \varpi} $, where $ \varpi \left( x \right) \sim {x^{\{2\tau+2\}}}/{\left( { - \ln x} \right)^\lambda } $ with $ \lambda  > 1 $. Assume (H2), (H3) and (H4). That is, $ H(x,y) $ is of $ C^{k} $ plus the above $ \varpi $ with $k= [2\tau+2] $. Then there is a solution $ x = u\left( \xi  \right),y = v\left( \xi  \right) $	of the following equation with the operator $ D: = \sum\limits_{\nu  = 1}^n {{\omega _\nu }\frac{\partial }{{\partial {\xi _\nu }}}} = \omega  \cdot {\partial _\xi } $
	\begin{equation}\notag
		Du = {H_y}\left( {u,v} \right),\;\;Dv =  - {H_x}\left( {u,v} \right)
	\end{equation}
	such that $ u\left( \xi  \right) - \xi  $ and $ v\left( \xi  \right) $ are of period $ 1 $ in all variables. In addition, letting
	\[{\varpi _1}\left( x \right) \sim \frac{1}{{{{\left( { - \ln x} \right)}^{\lambda  - 1}}}} \sim \varpi _{\mathrm{LH}}^{\lambda  - 1}\left( x \right),\]
	and
	\[	{\varpi _2}\left( x \right) \sim \left\{ \begin{aligned}
		&{\frac{1}{{{{\left( { - \ln x} \right)}^{\lambda  - 1}}}} \sim \varpi _{\mathrm{LH}}^{\lambda  - 1}\left( x \right)},&n-1<\tau  \in {\mathbb{N}^ + } \hfill, \\
		&{\frac{{{x^{\left\{ \tau  \right\}}}}}{{{{\left( { - \ln x} \right)}^\lambda }}} \sim {x^{\left\{ \tau  \right\}}}\varpi _{\mathrm{LH}}^\lambda \left( x \right)},&n-1<\tau  \notin {\mathbb{N}^ + } \hfill, \\
	\end{aligned}  \right.\]
	one has that $ u \in {C_{1 ,{\varpi _1}}}\left( {{\mathbb{R}^n},{\mathbb{R}^n}} \right) $ and $ v \circ {u^{ - 1}} \in {C_{[\tau+1] ,{\varpi _2}}}\left( {{\mathbb{R}^n},G} \right) $.
\end{theorem}

\subsubsection{H\"{o}lder plus generalized Logarithmic H\"{o}lder continuous case}
Actually, in view of the Dini type integrability condition (H1), one notices that the regularity for the given  Hamiltonian function could be further weakened, such as the generalized Logarithmic H\"{o}lder continuity in \eqref{mafan}. But the parameter $\lambda>0 $ needs to be  appropriately chosen, otherwise it may contradict \eqref{Dini}. We will only present the following Theorem \ref{GLHnew} for the case where Diophantine index $ \tau  $ is an integer for the sake of simplicity. One could deal with $ \tau \notin \mathbb{N}^+ $ similar to that in Theorem \ref{lognew}.

\begin{theorem}\label{GLHnew}
		Let $ n-1<\tau\in \mathbb{N}^+ $, and  $ H(x,y)\in C_{2\tau+2, {\varpi_{\mathrm{GLH}}^{\varrho,\lambda}}} $ with $\varrho\in \mathbb{N}^+, \lambda  > 1 $. That is, all $ (2\tau+2) $-order derivatives of $ H(x,y) $ admit generalized Logarithmic H\"{o}lder continuity:
\begin{equation}\label{mafan2}
	\varpi_{\mathrm{GLH}}^{\varrho,\lambda} \left( x \right) \sim \frac{1}{{(\ln (1/x))(\ln \ln (1/x)) \cdots {{(\underbrace {\ln  \cdots \ln }_\varrho (1/x))}^\lambda }}},\;\;x \to 0^+.
\end{equation}
		 Assume (H2), (H3) and (H4).  Then there is a solution $ x = u\left( \xi  \right),y = v\left( \xi  \right) $	of the following equation with the operator $ D: = \sum\limits_{\nu  = 1}^n {{\omega _\nu }\frac{\partial }{{\partial {\xi _\nu }}}} = \omega  \cdot {\partial _\xi } $
	\begin{equation}\notag
		Du = {H_y}\left( {u,v} \right),\;\;Dv =  - {H_x}\left( {u,v} \right)
	\end{equation}
	such that $ u\left( \xi  \right) - \xi  $ and $ v\left( \xi  \right) $ are of period $ 1 $ in all variables. Besides, the remaining regularity in Theorem \ref{theorem1} is $ u \in {C_{1 ,{\varpi _1}}}\left( {{\mathbb{R}^n},{\mathbb{R}^n}} \right) $ and $ v \circ {u^{ - 1}} \in {C_{\tau+1,{\varpi _2}}}\left( {{\mathbb{R}^n},G} \right) $ with modulus of continuity
	\begin{equation}\label{rem}
		{\varpi _1}\left( x \right) \sim {\varpi _2}\left( x \right) \sim \frac{1}{{{{(\underbrace {\ln  \cdots \ln }_\varrho (1/x))}^{\lambda  - 1}}}}.
	\end{equation}
\end{theorem}

\begin{remark}
	Particularly \eqref{mafan2} reduces to the Logarithmic H\"older's type $ \varpi_{\mathrm{LH}}^{\lambda}(x) \sim 1/{(-\ln x)^\lambda} $ with $ \lambda>1 $ as long as $ \varrho=1 $. As can be seen that, the remaining regularity in \eqref{rem} is much weaker than the initial regularity in \eqref{mafan2}, and it is indeed very weak if $ \lambda>1 $ is sufficiently close to $ 1 $ (but cannot degenerate to $ 1 $ by (H1)), or $ \varrho \in \mathbb{N}^+ $ is sufficiently large, because the explicit modulus of continuity in \eqref{rem} tends to $ 0 $ quite slowly as $ x \to 0^+ $.
\end{remark}

\subsection{Two explicit Hamiltonian systems of Logarithmic H\"{o}lder's type}\label{explicitexa}
To illustrate the wider applicability of our theorems, we shall present two explicit examples strictly beyond H\"older's type, involving \textit{non-H\"older corner's type} and even \textit{nowhere H\"older's type}. Note that  the H\"older plus Logarithmic H\"older regularity for $ H(x,y) $ in Theorem \ref{lognew} becomes simpler Logarithmic H\"older's type  for $ 2n<2\tau+2 \in \mathbb{N}^+ $ (because $ \{2 \tau+2 \}=0 $), we therefore consider the following \textit{critical settings} (with sharp differentiability) throughout this section.

Recall Theorem \ref{lognew}. Let $ n = 2,\tau  = 2, k = 6=[2\tau+2], {\alpha _ * } > 0,\lambda  > 1$ and $M > 0 $ be  given. Assume that $ \left( {x,y} \right) \in {\mathbb{T}^2} \times G $ with $ G := \{ {y \in {\mathbb{R}^2}:\left| y \right| \leq 1} \} $, and the universal Diophantine frequency $ \omega  = {\left( {{\omega _1},{\omega _2}} \right)^{\top}} \in \mathbb{R}^2 $ satisfies
\begin{equation}\notag
	| {\langle {{\tilde k},\omega } \rangle } | \geq \alpha_ *{ | {{\tilde k }} |}^{-2} ,\;\;\forall 0 \ne \tilde k \in {\mathbb{Z}^2},\;\;|\tilde k|: = |k_1|+|k_2|,
\end{equation}
i.e., with full Lebesgue measure.
\subsubsection{Non-H\"older corner's type}\label{SUBSUB1}
 Now we shall construct a function for finite smooth  perturbations via \textit{non-H\"older corner's type}, whose regularity is $ C^6 $ plus Logarithmic H\"older's type $ \varpi_{\mathrm{LH}}^{\lambda}(r) \sim 1/(-\ln r)^{\lambda} $ with index $ \lambda>1 $. In other words, the highest order derivative  admits exact  Logarithmic H\"older regularity at a non-H\"older corner, but it may have good regularity away from this corner, such as being  sufficiently smooth.  Namely, define
\begin{equation}\notag
	\mathcal{P}(r): =  \left\{ \begin{aligned}
		&{{\int_0^r { \cdots \int_0^{{s_2}} {\frac{1}{{{{(1 - \ln \left| {{s_1}} \right|)}^\lambda }}}{\rm d}{s_1} \cdots {\rm d}{s_6}} }}},&{0 < \left| r \right| \leq 1} \hfill, \\
		&{0},&{r=0} \hfill. \\
	\end{aligned}  \right.
\end{equation}
Obviously $ \mathcal{P}(r)\in C_{6,\varpi_{\mathrm{LH}}^{\lambda}} ([-1,1])$.
Let us consider the perturbed Hamiltonian function below with some constant $ 0 < {\varepsilon } < {\varepsilon ^ * } $ sufficiently small ($ {\varepsilon ^ * } $ depends on the constants given before):
\begin{equation}\label{HH}
	\mathcal{H}(x,y,\varepsilon ) = {\omega _1}{y_1} + {\omega _2}{y_2} + \frac{1}{M}(y_1^2 + y_2^2) + \varepsilon \left( {\sin (2\pi {x_1}) + \sin (2\pi {x_2}) + \mathcal{P}({y_1}) + \mathcal{P}\left( {{y_2}} \right)} \right).
\end{equation}
At this point, we have
\begin{align*}
	\left| {{{\left( {\int_{{\mathbb{T}^2}} {{\mathcal{H}_{yy}}\left( {\xi ,0} \right){\rm d}\xi } } \right)}^{ - 1}}} \right| &= \left| {{{\left( {\int_{{\mathbb{T}^2}} {\left( {\begin{array}{*{20}{c}}
								{2{M^{ - 1}}}&0 \\
								0&{2{M^{ - 1}}}
						\end{array}} \right){\rm d}\xi } } \right)}^{ - 1}}} \right| \notag \\
	&= \left| {\left( {\begin{array}{*{20}{c}}
				{{2^{ - 1}}M}&0 \\
				0&{{2^{ - 1}}M}
		\end{array}} \right)} \right| \leq M <  + \infty .
\end{align*}
In addition, one can verify that $ \mathcal{H}(x,y,\varepsilon) \in {C_{6,{\varpi_{\mathrm{LH}}^{\lambda}}}}( {{\mathbb{T}^2} \times G} ) $ with $ \varpi_{\mathrm{LH}}^{\lambda}(r) \sim 1/(-\ln r)^{\lambda} $.

However, for $ \tilde \alpha  = {\left( {0,0,6,0} \right)^{\top}} $ with $ \left| {\tilde \alpha } \right| = 6 = k $, we have
\[\left| {{\partial ^{\tilde \alpha }}\mathcal{H}\left( {{{\left( {0,0} \right)}^{\top}},{{\left( {{y_1},0} \right)}^{\top}},\varepsilon } \right) - {\partial ^{\tilde \alpha }}\mathcal{H}\left( {{{\left( {0,0} \right)}^{\top}},{{\left( {0,0} \right)}^{\top}},\varepsilon } \right)} \right| = \frac{\varepsilon }{{{{(1 - \ln \left| {{y_1}} \right|)}^\lambda }}} \geq \varepsilon {c_{\lambda ,\ell }}{\left| {{y_1}} \right|^\ell }\]
for any $ 0<\ell\leq1 $, where $ c_{\lambda ,\ell } >0$ is a constant that only depends on $ \lambda $ and $ \ell $. This implies that $ \mathcal{H}(x,y,\varepsilon) \notin {C_{6,{\varpi _{\mathrm{H}}^\ell}}}( {{\mathbb{T}^2} \times G} )  $ with $ {\varpi _{\mathrm{H}}^\ell}(r) \sim {r^\ell } $, i.e., $ \mathcal{H}(x,y,\varepsilon) \notin C^{6+\ell}( {{\mathbb{T}^2} \times G} ) $ with any $ 0<\ell \leq 1 $, because $ \varpi_{\mathrm{LH}}^{\lambda} $ is strictly weaker than $ \varpi _{\mathrm{H}}^\ell $, see also Remark  \ref{strict} and Corollary \ref{V37-Re2.3}.

In other words, the highest order  derivatives (of order $ k=6 $) of $ H(x,y) $ in \eqref{HH} can be rigorously proved to be Logarithmic H\"{o}lder continuous with index $ \lambda>1 $, but not any  H\"{o}lder's type. Therefore, the finite smooth KAM theorems via classical H\"{o}lder continuity cannot be applied. But, all the assumptions of  Theorem \ref{lognew} can be verified to be satisfied, then the invariant torus persists, and the frequency $ \omega  = {\left( {{\omega _1},{\omega _2}} \right)^{\top}} $ for the unperturbed system can remain unchanged. Moreover, the remaining regularity for mappings $ u $ and $ v \circ u^{-1} $ in Theorem \ref{lognew} could also be determined as $ u \in {C_{1 ,{\varpi _\mathrm{LH}^{\lambda-1}}}}\left( {{\mathbb{R}^n},{\mathbb{R}^n}} \right) $   and $ v \circ {u^{ - 1}} \in {C_{3 ,{\varpi _\mathrm{LH}^{\lambda-1}}}}\left( {{\mathbb{R}^n},G} \right) $, where $ \varpi _\mathrm{LH}^{\lambda-1}(r)\sim 1/(-\ln r)^{\lambda-1} $. More precisely,  $ u $ is at least $ C^1 $, while $ v \circ u^{-1}$ is least $ C^3 $, and the higher regularity for them is still not any H\"older's type, but Logarithmic H\"older one with index $ \lambda-1 $, i.e., lower than the original index $ \lambda>1 $, this is because the small divisors causes the loss of regularity. In KAM terminology, the regularity of  invariant torus with frequency-preserving is $ C^3 $ plus $ (\lambda-1) $-Logarithmic H\"older, while the conjugation from the dynamic on it to the linear flow is $ C^1 $ plus $ (\lambda-1) $-Logarithmic H\"older, see Remark \ref{xiangxi}.

 \subsubsection{Nowhere H\"older's type}
 Although an explicit irregular case has been discussed in Section \ref{SUBSUB1}, it is not universal enough. Notice that the non-integrable part and perturbations for Hamiltonian should at least be of $ C^6 $ under our results (see Theorem \ref{theorem1} and Remark \ref{RemarkM2}). However, if perturbations are selected at random from the function space $ C^6 $, then the derivatives of $ 6 $-order are likely to be \textit{nowhere-differentiable}  (note that the example given in Section \ref{SUBSUB1} is $ C^\infty $ except for the corner point), or even \textit{nowhere H\"older} from the point of view of Baire category. More precisely, perturbations may be convergent trigonometric series composed of \textit{high frequency oscillations}. This is also one of our motivations for establishing KAM in the sense of modulus of continuity. We shall focus on this case below, aiming to  show the strength of our results.

 Given $ 0<q<\frac{1}{3} $ and $ \lambda>1 $. Define a sequence $ \{\mathcal{Q}_n\}_{n \in \mathbb{N}^+} $ that satisfies $ \left\{ {{\mathcal{Q}_{n + 1}}\mathcal{Q}_n^{-1},{\mathcal{Q}_n}} \right\}_{n \in \mathbb{N}^+}  \subseteq  10{\mathbb{N}^ + } $ and  $ \mathcal{Q}_n=\mathcal{O}^{\#}\left(\exp(\Theta^n)\right) $ with $ \Theta=\left({\frac{1}{q}}\right)^{\frac{1}{\lambda}}>1 $. 	Then the $ 1 $-periodic function 
 	\begin{equation}\label{mathscrP}
 		\mathscr{P}\left( x \right): = \sum\limits_{n = 1}^\infty  {{q^n}\sin \left( {\pi {\mathcal{Q}_n}x} \right)} ,\;\;x \in \mathbb{R}
 	\end{equation}
is well defined. We shall construct a Hamiltonian system using $ \mathscr{P} $ as the highest order derivatives for the perturbation. To this end,  consider the following Hamiltonian with \textit{high frequency oscillation perturbation}:
 \begin{align}
 	\mathscr{H}(x,y,\varepsilon ) &= {\omega _1}{y_1} + {\omega _2}{y_2} + \left\langle {\mathscr{A}\left( {{x_1},{x_2}} \right){{\left( {{y_1},{y_2}} \right)}^\top},{{\left( {{y_1},{y_2}} \right)}^\top}} \right\rangle  \notag \\
 	 &\;\;\;\;+ \varepsilon \left( \sum\limits_{n = 1}^\infty  {\frac{{{q^n}}}{{\mathcal{Q}_n^6}}\sin \left( {\pi {\mathcal{Q}_n}{x_1}} \right)}  + \sum\limits_{n = 1}^\infty  {\frac{{{q^n}}}{{\mathcal{Q}_n^6}}\sin \left( {\pi {\mathcal{Q}_n}{y_2}} \right)}  \right)\notag \\
 \label{HHH}	:&=\mathscr{N}_\mathscr{H}+\varepsilon \mathscr{P}_\mathscr{H},
 \end{align}
 provided with some constant $ 0 < {\varepsilon } < {\varepsilon ^ * } $ sufficiently small, and $ {\varepsilon ^ * } $ depends on the constants given before. Here the matrix $ \mathscr{A}(x_1,x_2) $ could  only be of class $ C_{6,\varpi_{\mathrm{LH}}^\lambda} $ (recall Remark \ref{RemarkM2}), and assume that (H2) holds. At this case, the unperturbed part $ \mathscr{N}_{\mathscr{H}} $ of $ \mathscr{H} $ is indeed $ C_{6,\varpi_{\mathrm{LH}}^\lambda} $ finitely smooth and non-integrable. Let us point out that, the non-integrable  part could have weaker regularity, i.e., 
 \begin{align*}
 		\tilde{\mathscr{H}}(x,y,\varepsilon ) &= {\omega _1}{y_1} + {\omega _2}{y_2} + \left\langle \left(\mathscr{A} \left(x_1,x_2\right)+{\tilde{\varepsilon}}{\tilde{\mathscr{A}}} \left(x_1,x_2\right)\right){{{\left( {{y_1},{y_2}} \right)}^\top},{{\left( {{y_1},{y_2}} \right)}^\top}} \right\rangle\notag \\
 	&\;\;\;\;+ \varepsilon \left( \sum\limits_{n = 1}^\infty  {\frac{{{q^n}}}{{\mathcal{Q}_n^6}}\sin \left( {\pi {\mathcal{Q}_n}{x_1}} \right)}  + \sum\limits_{n = 1}^\infty  {\frac{{{q^n}}}{{\mathcal{Q}_n^6}}\sin \left( {\pi {\mathcal{Q}_n}{y_2}} \right)}  \right) \notag \\
 	&\;\;\;\;-{\tilde{\varepsilon}}\left\langle {\mathscr{A}\left( {{x_1},{x_2}} \right){{\left( {{y_1},{y_2}} \right)}^\top},{{\left( {{y_1},{y_2}} \right)}^\top}} \right\rangle ,
 \end{align*}
 where $ {\tilde{\varepsilon}}\tilde{\mathscr{A}} $ and $ {\tilde{\varepsilon}}\bar{\mathscr{A}} $ are small perturbations of $ \mathscr{A} $, admitting weaker regularity than $ C_{6,\varpi_{\mathrm{LH}}^\lambda} $, e.g., $ C^5 $, and $ \tilde{\mathscr{A}}-\bar{\mathscr{A}} $ \textit{must} be of at least $ C_{6,\varpi_{\mathrm{LH}}^\lambda} $. The analyzability of this  case  is due to our technique for dealing with finite smooth KAM in this paper, but obviously, the new perturbation of $ \tilde{\mathscr{H}} $ is quite special.

Back to our concern on Hamiltonian \eqref{HHH}.  Now, the universal frequency-preserving KAM persistence and remaining regularity same as that in Section \ref{SUBSUB1} could be obtained if one proves that the perturbation is also $ C_{6,\varpi_{\mathrm{LH}}^\lambda} $. However, this is not simple. The following proposition verifies this claim, but simultaneously shows that it is indeed nowhere H\"older continuous (and therefore, the H\"older type KAM tools cannot be used at all, at any point). As you will see later, the technique used here is somewhat similar to that in Abstract Iterative Theorem \ref{t1}.

 \begin{proposition}
 The perturbation part $ \varepsilon \mathscr{P}_\mathscr{H} $ of Hamiltonian $ \mathscr{H} $ in \eqref{HHH} is at least $ C^6 $, and the regularity for $ 6 $-order derivatives is $ \lambda $-Logarithmic  H\"older's type, but is nowhere H\"older continuous.
 \end{proposition}
 \begin{proof}
 By direct derivation, it suffices to verify that $ \mathscr{P}\left( x \right) $ defined in \eqref{mathscrP}  is $ \lambda $-Logarithmic  H\"older continuous, but is nowhere H\"older continuous.
 Obivously $ \mathscr{P}\left( x \right) $ is continuous on $ \mathbb{R} $ because 
 \[\sum\limits_{n = 1}^\infty  {\left| {{q^n}\sin \left( {\pi {\mathcal{Q}_n}x} \right)} \right|}  \leq \sum\limits_{n = 1}^\infty  {{q^n}}  <  + \infty .\]
 Then it immediately follows from Remark \ref{rema666} and the $ 1 $-periodicity of $ \mathscr{P} $ that $ \mathscr{P} $ automatically admits a modulus of continuity $ \varpi_{\mathscr{P}} $. We first prove the $ \lambda $-Logarithmic  H\"older continuity, i.e., $ \varpi_{\mathscr{P}}  $  is weaker than $ \varpi_{\rm LH}^\lambda \sim 1/{\left( { - \ln x} \right)^\lambda }$ for $ \lambda>1 $. For every fixed $ 0<h<1 $, define (it is indeed well defined because $ \mathcal{Q}_n \to +\infty $ as $ n \to +\infty $)
 \begin{equation}\label{huaNdingyi}
 	\mathscr{N}_h: = \max \left\{ {\mathscr{N}_h \in {\mathbb{N}^ + }:h {\mathcal{Q}_{\mathscr{N}_h}} \leq 1} \right\}.
 \end{equation}
 Now one derives
 \begin{equation}\label{jgjgz}
 	{q^{{\mathscr{N}_h}}} \leq C\left( {q,\lambda } \right)\varpi _{\rm LH}^\lambda \left( h \right),\;\; \forall 0<h \ll 1 
 \end{equation}
 from the assumptions on $ \{\mathcal{Q}_n\}_{n \in \mathbb{N}^+} $, where $ C\left( {q,\lambda } \right)>0 $ is a generic  constant that only depends on $ q$ and $ \lambda $. Therefore, for $ 0<h<1 $ sufficiently small, we obtain that
 \begin{align}
 	\left| {\mathscr{P}\left( {x + h} \right) - \mathscr{P}\left( x \right)} \right| &= \left| {\sum\limits_{n = 1}^\infty  {{q_n}\left( {\sin \left( {\pi {\mathcal{Q}_n}x + \pi {\mathcal{Q}_n}h} \right) - \sin \left( {\pi {\mathcal{Q}_n}x} \right)} \right)} } \right|\notag \\
 	& \leq \sum\limits_{n = 1}^{\mathscr{N}_h} { + \sum\limits_{n = \mathscr{N}_h + 1}^\infty  {{q_n}\left| {\sin \left( {\pi {\mathcal{Q}_n}x + \pi {\mathcal{Q}_n}h} \right) - \sin \left( {\pi {\mathcal{Q}_n}x} \right)} \right|} } \notag\\
 	\label{GZ1}	& \leq h\pi \sum\limits_{n = 1}^{\mathscr{N}_h} {{\mathcal{Q}_n}{q_n}}  + 2\sum\limits_{n = \mathscr{N}_h + 1}^\infty  {{q_n}} \\
 	& = h\pi {Q_{{\mathscr{N}_h}}}{q^{{\mathscr{N}_h}}}\sum\limits_{n = 1}^{{\mathscr{N}_h}} {\frac{{{Q_n}}}{{{Q_{{\mathscr{N}_h}}}}} \cdot \frac{{{q^n}}}{{{q_{{\mathscr{N}_h}}}}}}  + \sum\limits_{n = {\mathscr{N}_h} + 1}^\infty  {{q^n}}\notag \\
 	\label{GZ2}& \leq h\pi {Q_{{\mathscr{N}_h}}}{q^{{\mathscr{N}_h}}}\sum\limits_{n = 1}^{{\mathscr{N}_h}} {{{\left( {\frac{q}{{10}}} \right)}^{n - {\mathscr{N}_h}}}}  + \sum\limits_{n = {\mathscr{N}_h} + 1}^\infty  {{q^n}} \\
 	& \leq \frac{{10\pi }}{{10 - q}} \cdot h{Q_{{\mathscr{N}_h}}} \cdot {q^{{\mathscr{N}_h}}} + \frac{q}{{1 - q}}{q^{{\mathscr{N}_h}}}\notag\\
 	\label{GZ3}& \leq C\left( {q,\lambda } \right){q^{{\mathscr{N}_h}}} \\
 	\label{GZ4}& \leq C\left( {q,\lambda } \right)\varpi _{\rm LH}^\lambda \left( h \right),\;\;\forall x \in 
 	\mathbb{R},
 \end{align}
 where \eqref{GZ1} uses the Mean Value Theorem and $ \left| {\sin x} \right|\leq1 $, \eqref{GZ2} is due to $ \left\{ {{\mathcal{Q}_{n + 1}}}\mathcal{Q}_n^{-1} \right\}_{n \in \mathbb{N}^+}  \subseteq  10{\mathbb{N}^ + } $, \eqref{GZ3} employs the definition of $ \mathscr{N}_h $ in \eqref{huaNdingyi}, and finally \eqref{GZ4} follows from the estimate \eqref{jgjgz}.  This gives the claim asserting that $ \mathscr{P} (x)$ is $ \lambda $-Logarithmic H\"older with respect to $ \lambda>1 $.
 
 Next we will show that $ \mathscr{P}(x) $ is nowhere H\"older continuous on $ \mathbb{R} $. Let   $ \alpha>0 $ be arbitrary given, and consider $ \varpi_{\rm H}^{\alpha}\sim x^\alpha $ (arbitrary H\"older modulus of continuity). 	Fix $ x>0 $, and choose $ m=m_{x,\alpha} \in \mathbb{N}^+ $ sufficiently large such that $ {\mathcal{Q}_m}x \geq 1 $, and
 \begin{equation}\label{fanjieGZ}
 	{\mathcal{Q}_m} \geq {\left( {\varpi _{\rm H}^\alpha \left( {\frac{{1 - 3q}}{{2\left( {1 - q} \right)}}\frac{{{q^m}}}{m}} \right)} \right)^{ - 1}}.
 \end{equation}
 Note that this is achievable because $ \mathcal{Q}_m $ is of the superexponential order.
 Now it reads  $ {\mathcal{Q}_m}x = {\mathcal{N}_m} + {\mathcal{R}_m} $ with $ {\mathcal{N}_m} \in {\mathbb{N}^ + }$ and $0 < \left| {{\mathcal{R}_m}} \right| \leq {2^{ - 1}} $. Taking 
 \begin{equation}\label{upsilondy}
 	\upsilon_{x,m}   = \mathcal{Q}_m^{ - 1}\left( { - \operatorname{sgn} \left( {{\mathcal{R}_m}} \right){2^{ - 1}} - {\mathcal{R}_m}} \right)  \to 0
 \end{equation}
 as $ m \to \infty $ yields that $ {2^{ - 1}} \leq \left| {\upsilon_{x,m}} \right|{\mathcal{Q}_m} \leq 1 $. Then  we get
 \begin{align}
 	\frac{{\mathscr{P}\left( {x + \upsilon_{x,m}} \right) - \mathscr{P}\left( x \right)}}{{{\varpi_{\rm H}^{\alpha}}\left( {\left| {\upsilon_{x,m}} \right|} \right)}}	= &\frac{{{q^m}}}{{{\varpi_{\rm H}^{\alpha}}\left( {\left| {\upsilon_{x,m}} \right|} \right)}}\left( {\sin \left( {\pi {\mathcal{Q}_m}\left( {x + \upsilon_{x,m}} \right)} \right) - \sin \left( {\pi {\mathcal{Q}_m}x} \right)} \right) \notag\\
 	&+ \sum\limits_{n = 1}^{m - 1} {\frac{{{q^n}}}{{{\varpi_{\rm H}^{\alpha}}\left( {\left| {\upsilon_{x,m}} \right|} \right)}}\left( {\sin \left( {\pi {\mathcal{Q}_n}\left( {x + \upsilon_{x,m}} \right)} \right) - \sin \left( {\pi {\mathcal{Q}_n}x} \right)} \right)} \notag\\
 	& + \sum\limits_{n = m + 1}^\infty  {\frac{{{q^n}}}{{{\varpi_{\rm H}^{\alpha}}\left( {\left| {\upsilon_{x,m}} \right|} \right)}}\left( {\sin \left( {\pi {\mathcal{Q}_n}\left( {x + \upsilon_{x,m}} \right)} \right) - \sin \left( {\pi {\mathcal{Q}_n}x} \right)} \right)} \notag \\
 	\label{sanduan}	: = &{\mathscr{S}_1} + {\mathscr{S}_2} + {\mathscr{S}_3}.
 \end{align}
 Firstly, direct calculation as well as \eqref{upsilondy} gives
 \begin{align}
 	\left| {{\mathscr{S}_1}} \right| &= \frac{{{q^m}}}{{{\varpi_{\rm H}^{\alpha}}\left( {\left| {\upsilon_{x,m}} \right|} \right)}}\left| {\sin \left( {\pi {\mathcal{Q}_m}\left( {x + \upsilon_{x,m}} \right)} \right) - \sin \left( {\pi {\mathcal{Q}_m}x} \right)} \right|\notag\\
 	& = \frac{{{q^m}}}{{{\varpi_{\rm H}^{\alpha}}\left( {\left| {\upsilon_{x,m}} \right|} \right)}}\left| {{{\left( { - 1} \right)}^{{\mathcal{N}_m} + 1}}\left( {\operatorname{sgn} \left( {{\mathcal{R}_m}} \right) + \sin \left( {\pi {\mathcal{R}_m}} \right)} \right)} \right|\notag\\
 	& = \frac{{{q^m}}}{{{\varpi_{\rm H}^{\alpha}}\left( {\left| {\upsilon_{x,m}} \right|} \right)}}\left| {\operatorname{sgn} \left( {{\mathcal{R}_m}} \right) + \sin \left( {\pi {\mathcal{R}_m}} \right)} \right|\notag\\
 	\label{sanduan1}	& \geq \frac{{{q^m}}}{{{\varpi_{\rm H}^{\alpha}}\left( {\mathcal{Q}_m^{ - 1}} \right)}}.
 \end{align}
 Secondly,  by applying the Mean Value Theorem we obtain that
 \begin{align}
 	\left| {{\mathscr{S}_2}} \right| &\leq \sum\limits_{n = 1}^{m - 1} {\frac{{{q^n}}}{{{\varpi_{\rm H}^{\alpha}}\left( {\left| {\upsilon_{x,m}} \right|} \right)}}\left| {\sin \left( {\pi {\mathcal{Q}_n}\left( {x + \upsilon_{x,m}} \right)} \right) - \sin \left( {\pi {\mathcal{Q}_n}x} \right)} \right|} \notag\\
 	&\leq \sum\limits_{n = 1}^{m - 1} {\frac{{{q^n}}}{{{\varpi_{\rm H}^{\alpha}}\left( {\left| {\upsilon_{x,m}} \right|} \right)}} \cdot \pi {\mathcal{Q}_n}\left| {\upsilon_{x,m}} \right|} \notag\\
 	& \leq \left( {\pi \sum\limits_{n = 1}^\infty  {{q^n}} } \right)\frac{{{\mathcal{Q}_{m - 1}}\mathcal{Q}_m^{ - 1}}}{{{\varpi_{\rm H}^{\alpha}}\left( {\mathcal{Q}_m^{ - 1}} \right)}}\notag\\
 	&= \frac{{{\mathcal{Q}_{m - 1}}\mathcal{Q}_m^{ - 1}}}{{{\varpi_{\rm H}^{\alpha}}\left( {\mathcal{Q}_m^{ - 1}} \right)}}\cdot \frac{{\pi q}}{{1 - q}}\notag\\
 	\label{sanduan2}&\leq \frac{q^m}{2{{\varpi_{\rm H}^{\alpha}}\left( {\mathcal{Q}_m^{ - 1}} \right)}},
 \end{align}
 where we use the following fact due to $ \left\{ {{\mathcal{Q}_{n + 1}}}\mathcal{Q}_n^{-1} \right\}_{n \in \mathbb{N}^+}  \subseteq  10{\mathbb{N}^ + } $:  
 \begin{equation}\label{3pi}
 	{\mathcal{Q}_{m - 1}}\mathcal{Q}_m^{ - 1}\left( {\frac{{\pi q}}{{1 - q}}} \right) \leq \frac{{\pi q}}{{10\left( {1 - q} \right)}} < \frac{{3\pi q}}{{20}} \leq \frac{1}{2}{q^m},\;\;\forall m \in {\mathbb{N}^ + }.
 \end{equation}
 Thirdly, note that $ {\mathcal{Q}_n}\mathcal{Q}_m^{ - 1} \in 10\mathbb{N}^+ $. Then it follows that
 \begin{align*}
 	\sin \left( {\pi {\mathcal{Q}_n}\left( {x + \upsilon_{x,m}} \right)} \right) &= \sin \left( {\pi {\mathcal{Q}_n}\mathcal{Q}_m^{ - 1} \cdot {\mathcal{Q}_m}\left( {x + \upsilon_{x,m}} \right)} \right)\\
 	& = \sin \left( {\pi {\mathcal{Q}_n}\mathcal{Q}_m^{ - 1} \cdot \left( {{\mathcal{N}_m} - {2^{ - 1}}\operatorname{sgn} \left( {{\mathcal{R}_m}} \right)} \right)} \right)\\
 	& = 0,\;\;\forall n \geq m + 1,
 \end{align*}
 which leads to
 \begin{align}
 	\left| {{\mathscr{S}_3}} \right| &\leq \sum\limits_{n = m + 1}^\infty  {\frac{{{q^n}}}{{{\varpi_{\rm H}^{\alpha}}\left( {\left| {\upsilon_{x,m}} \right|} \right)}}\left| {\sin \left( {\pi {\mathcal{Q}_n}\left( {x + \upsilon_{x,m}} \right)} \right) - \sin \left( {\pi {\mathcal{Q}_n}x} \right)} \right|} \notag\\
 	& = \sum\limits_{n = m + 1}^\infty  {\frac{{{q^n}}}{{{\varpi_{\rm H}^{\alpha}}\left( {\left| {\upsilon_{x,m}} \right|} \right)}}\left| {\sin \left( {\pi {\mathcal{Q}_n}x} \right)} \right|} \notag\\
 	& \leq \frac{1}{{{\varpi_{\rm H}^{\alpha}}\left( {\left| {\upsilon_{x,m}} \right|} \right)}}\sum\limits_{n = m + 1}^\infty  {{q^n}} \notag\\
 	\label{sanduan3}	&\leq \frac{1}{{{\varpi_{\rm H}^{\alpha}}\left( {\mathcal{Q}_m^{ - 1}} \right)}}  \cdot \frac{{{q^m}}}{{1 - q}} .
 \end{align}
 Finally, substituting \eqref{sanduan1}, \eqref{sanduan2}, \eqref{sanduan3} into \eqref{sanduan} and using \eqref{fanjieGZ} yields that
 \begin{align*}
 	\frac{{\left| {\mathscr{P}\left( {x + \upsilon_{x,m} } \right) - \mathscr{P}\left( x \right)} \right|}}{{{\varpi_{\rm H}^{\alpha}}\left( {\left| {\upsilon_{x,m}} \right|} \right)}}	&\geq \left| {{\mathscr{S}_1}} \right| - \left| {{\mathscr{S}_2}} \right| - \left| {{\mathscr{S}_3}} \right|\\
 	&\geq \frac{{{q^m}}}{{{\varpi_{\rm H}^{\alpha}}\left( {\mathcal{Q}_m^{ - 1}} \right)}} - \frac{{{q^m}}}{2{{\varpi_{\rm H}^{\alpha}}\left( {\mathcal{Q}_m^{ - 1}} \right)}} - \frac{1}{{{\varpi_{\rm H}^{\alpha}}\left( {\mathcal{Q}_m^{ - 1}} \right)}}  \cdot \frac{{{q^m}}}{{1 - q}} \\
 	&= \frac{1}{{{\varpi_{\rm H}^{\alpha}}\left( {\mathcal{Q}_m^{ - 1}} \right)}}\left(  {\frac{{1 - 3q}}{{2\left( {1 - q} \right)}}} \right)q^m\\
 	&\geq m .
 \end{align*}
 In other words, for every $ x\in \mathbb{R} $ (note that $ \mathscr{P} $ is $ 1 $-periodic), one could construct a sequence $ \{\upsilon_{x,m}\}_{m \in \mathbb{N}^+} $ such that
 \[\mathop {\underline {\lim } }\limits_{m \to \infty } 	\frac{{\left| {\mathscr{P}\left( {x + \upsilon_{x,m} } \right) - \mathscr{P}\left( x \right)} \right|}}{{{\varpi_{\rm H}^{\alpha}}\left( {\left| {\upsilon_{x,m}} \right|} \right)}} =  + \infty. \]
 This implies that $ \mathscr{P}(x)$ is nowhere H\"older continuous because the H\"older index $ 0<\alpha<1 $ could be arbitrary chosen, and we therefore prove the proposition.
\end{proof}

\section{Proof of Main Theorem \ref{theorem1}}\label{section4}

Now let us prove  Theorem \ref{theorem1} by separating two subsections, namely frequency-preserving KAM persistence (Section \ref{KAM}) and remaining  regularity (Section \ref{furtherregularity}) for KAM torus as well as conjugation. For the former, the overall process is similar to that in \cite{salamon}, but the key points to weaken the H\"older regularity to only modulus of continuity are using Jackson type approximation Theorem \ref{Theorem1} and proving the uniform convergence of the transformation mapping, that is, the convergence of the upper bound series (see \eqref{dao} and \eqref{dao2}). As we will see later, the Dini type integrability condition  (H1) is crucial for the boundedness. As to the latter, we have to establish a more general regularity iterative theorem (Theorem \ref{t1}), which does not seem simple since the resulting regularity might be somewhat  complicated due to asymptotic analysis.

\subsection{Frequency-preserving KAM persistence}\label{KAM}
The proof of the frequency-preserving KAM persistence is organized as follows. Firstly, we construct a series of analytic approximation functions $ H^\nu(x,y) $ of $ H(x,y)$ by using  Theorem \ref{Theorem1} and considering (H1) and (H2). Secondly, we shall construct a sequence of frequency-preserving analytic and symplectic transformations $ \psi^\nu $ by induction. According to (H2), (H3) and (H4), the first step of induction is established by applying  Theorem \ref{appendix} in Section \ref{Appsalamon} (or Theorem 1 in \cite{salamon}). Then, combining with weak homogeneity and certain specific estimates we complete the proof of induction and obtain the uniform convergence of the composite transformations. Finally, in the light of (H5), the regularity of the KAM torus as well as the conjugation is guaranteed by  Theorem \ref{t1}. We shall emphasize that the unperturbed part of Hamiltonian $ H(x,y) $ could be non-integrable and finitely smooth due to our strategy.\vspace{3mm}
\\
{\textbf{Step1:}} In view of   Theorem \ref{Theorem1} (we have assumed that the modulus of continuity $ \varpi $ admits semi separability and thus Theorem \ref{Theorem1} could be directly  applied here), one could approximate $ H(x, y) $ by a sequence of real analytic functions $ H_\nu(x, y) $ for $ \nu \geq 0 $ in the strips
\[\left| {\operatorname{Im} x} \right| \leq {r_\nu },\;\;\left| {\operatorname{Im} y} \right| \leq {r_\nu },\;\;{r_\nu }: = {2^{ - \nu }}\varepsilon \]
around $ \left| {\operatorname{Re} x} \right| \in {\mathbb{T}^n},\left| {\operatorname{Re} y} \right| \leq \rho , $ such that
\begin{equation}\label{T1-3}
	\begin{aligned}
		\left| {{H^\nu }\left( z \right) - \sum\limits_{\left| \alpha  \right| \leq k} {{\partial ^\alpha }H\left( {\operatorname{Re} z} \right)\frac{{{{\left( {\mathrm{i}\operatorname{Im} z} \right)}^\alpha }}}{{\alpha !}}} } \right| \leq{}& {c_1}{\left\| H \right\|_\varpi }r_\nu ^k\varpi \left( {{r_\nu }} \right),\\
		\left| {H_y^\nu \left( z \right) - \sum\limits_{\left| \alpha  \right| \leq k-1} {{\partial ^\alpha }{H_y}\left( {\operatorname{Re} z} \right)\frac{{{{\left( {\mathrm{i}\operatorname{Im} z} \right)}^\alpha }}}{{\alpha !}}} } \right| \leq{}& {c_1}{\left\| H \right\|_\varpi }r_\nu ^{k - 1}\varpi \left( {{r_\nu }} \right), \\
		\left| {H_{yy}^\nu \left( z \right) - \sum\limits_{\left| \alpha  \right| \leq k-2} {{\partial ^\alpha }{H_{yy}}\left( {\operatorname{Re} z} \right)\frac{{{{\left( {\mathrm{i}\operatorname{Im} z} \right)}^\alpha }}}{{\alpha !}}} } \right| \leq{}& {c_1}{\left\| H \right\|_\varpi }r_\nu ^{k - 2}\varpi \left( {{r_\nu }} \right)
	\end{aligned}
\end{equation}
for $ \left| {\operatorname{Im} x} \right| \leq {r_\nu },\;\left| {\operatorname{Im} y} \right| \leq {r_\nu } $, and $ c_1=c(n,k) $ is the constant provided in \eqref{3.2}.

Fix $ \theta  = 1/\sqrt 2  $. In what follows,  we will construct a sequence of real analytic symplectic transformations $ z=(x,y),\zeta=(\xi,\eta),z = {\phi ^\nu }\left( \zeta  \right) $ of the form
\begin{equation}\label{bhxs}
	x = {u^\nu }\left( \xi  \right),\;\;y = v^{\nu}\left( \xi  \right) + (u_\xi ^\nu)^{\top}{\left( \xi  \right)^{ - 1}}\eta
\end{equation}
by induction, such that $ {u^\nu }\left( \xi  \right) - \xi  $ and $ {v^\nu }\left( \xi  \right) $ are of period $ 1 $ in all variables, and $ {\phi ^\nu } $ maps the strip $ \left| {\operatorname{Im} \xi } \right| ,\left| \eta  \right| \leq \theta {r_{\nu  + 1}} $ into $ \left| {\operatorname{Im} x} \right| ,\left| y \right| \leq {r_\nu },\left| {\operatorname{Re} y} \right| \leq \rho  $, and the transformed Hamiltonian function $ 	{K^\nu }: = {H^\nu } \circ {\phi ^\nu } $ satisfies
\begin{equation}\label{qiudao}
	K_\xi ^\nu \left( {\xi ,0} \right) = 0,\;\;K_\eta ^\nu \left( {\xi ,0} \right) = \omega ,
\end{equation}
i.e., with prescribed universal Diophantine frequency-preserving. Namely by verifying certain conditions we obtain $ z=\psi^\nu(\zeta) $ of the form \eqref{bhxs} from Theorem \ref{appendix} by induction, mapping $ \left| {\operatorname{Im} \xi } \right|,\left| \eta  \right| \leq {r_{\nu  + 1}} $ into $ \left| {\operatorname{Im} x} \right|,\left| y \right| \leq \theta {r_\nu } $, and $ \psi^\nu \left( {\xi ,0} \right) - \left( {\xi ,0} \right) $ is of period $ 1 $, and \eqref{qiudao} holds. Here we denote $ \phi^\nu:=\phi^{\nu-1} \circ \psi^\nu$ with $ {\phi ^{ - 1}}: = \mathrm{id} $ (where $ \mathrm{id} $ denotes the $ 2n $-dimensional identity mapping and therefore $ {\phi ^0} = {\psi ^0} $). Further more, Theorem \ref{appendix} will lead to
\begin{align}
	\left| {{\psi ^\nu }\left( \zeta  \right) - \zeta } \right| &\leq c\left( {1 - \theta } \right)r_\nu ^{k - 2\tau  - 1}\varpi \left( {{r_\nu }} \right),\label{2.82}\\
	\left| {\psi _\zeta ^\nu\left( \zeta  \right) - \mathbb{I}} \right| &\leq cr_\nu ^{k - 2\tau  - 2}\varpi \left( {{r_\nu }} \right),\label{2.83}\\
	\left| {K_{\eta \eta }^\nu\left( \zeta  \right) - {Q^\nu}\left( \zeta  \right)} \right| &\leq cr_\nu ^{k - 2\tau  - 2}\varpi \left( {{r_\nu }} \right)/2M,\label{2.84}\\
	\left| {U_x^\nu \left( x \right)} \right| &\leq cr_\nu ^{k - \tau  - 1}\varpi \left( {{r_\nu }} \right),\label{2.85}
\end{align}
on $ \left| {\operatorname{Im} \xi } \right|,\left| \eta  \right|,\left| {\operatorname{Im} x} \right| \leq r_{\nu+1} $, where $ {S^\nu }\left( {x,\eta } \right) = {U^\nu }\left( x \right) + \left\langle {{V^\nu }\left( x \right),\eta } \right\rangle  $ is the generating function for $ {\psi ^\nu } $, and $ Q^\nu:=K_{\eta \eta}^{\nu-1} $, and $ \mathbb{I} $ denotes the $ 2n \times 2n $-dimensional identity mapping,  and
\begin{equation}\label{Q0}
	{Q^0}\left( z \right): = \sum\limits_{\left| \alpha  \right| \leq k - 2} {{\partial ^\alpha }{H_{yy}}\left( {\operatorname{Re} z} \right)\frac{{{{\left( {\mathrm{i}\operatorname{Im} x} \right)}^\alpha }}}{{\alpha !}}} .
\end{equation}
\\
{\textbf{Step2:}} Here we show that $ \psi^0=\phi^0 $ exists, and it admits the properties mentioned in Step 1. Denote
\begin{equation}\notag
	h(x) := H\left( {x,0} \right) - \int_{{\mathbb{T}^n}} {H\left( {\xi ,0} \right){\rm d}\xi } ,\;\;x \in {\mathbb{R}^n}.
\end{equation}
Then by the first term in \eqref{T1-2}, we have
\begin{equation}\label{pianh}
	\sum\limits_{\left| \alpha  \right| \leq k} {\left| {{\partial ^\alpha }h} \right|{\varepsilon ^{\left| \alpha  \right|}}}  < M{\varepsilon ^k}\varpi \left( \varepsilon  \right).
\end{equation}
Note that
\begin{align*}
	{H^0}\left( {x,0} \right) - \int_{{\mathbb{T}^n}} {{H^0}\left( {\xi ,0} \right){\rm d}\xi }  ={}& {H^0}\left( {x,0} \right) - \sum\limits_{\left| \alpha  \right| \leq k} {\partial _x^\alpha H\left( {\operatorname{Re} x,0} \right)\frac{{{{\left( {\mathrm{i}\operatorname{Im} x} \right)}^\alpha }}}{{\alpha !}}} \notag \\
	{}&+ \int_{{\mathbb{T}^n}} {\left( {H\left( {\xi ,0} \right) - {H^0}\left( {\xi ,0} \right)} \right){\rm d}\xi } \notag \\
	{}&+ \sum\limits_{\left| \alpha  \right| \leq k} {{\partial ^\alpha }h\left( {\operatorname{Re} x} \right)\frac{{{{\left( {\mathrm{i}\operatorname{Im} x} \right)}^\alpha }}}{{\alpha !}}} .
\end{align*}
Hence, for $ \left| {\operatorname{Im} x} \right| \leq \theta {r_0} = \theta \varepsilon  $, by using Theorem \ref{Theorem1}, Corollary \ref{coro1} and \eqref{pianh} we arrive at
\begin{align*}
	\left| {{H^0}\left( {x,0} \right) - \int_{{\mathbb{T}^n}} {{H^0}\left( {\xi ,0} \right){\rm d}\xi } } \right| &\leq 2{c_1}{\left\| H \right\|_\varpi }{\varepsilon ^k}\varpi \left( \varepsilon  \right) + M{\varepsilon ^k}\varpi \left( \varepsilon  \right)\notag \\
	&\leq c{\varepsilon ^k}\varpi \left( \varepsilon  \right) \notag \\
	&\leq c{\varepsilon ^{k - 2\tau  - 2}}\varpi \left( \varepsilon  \right) \cdot {\left( {\theta \varepsilon } \right)^{2\tau  + 2}}.
\end{align*}
Now consider the vector valued function $ 	f\left( x \right): = {H_y}\left( {x,0} \right) - \omega  $ for $ x \in {\mathbb{R}^n} $. In view of the second term in \eqref{T1-2}, we have
\begin{equation}\label{pianf}
	\sum\limits_{\left| \alpha  \right| \leq k - 1} {\left| {{\partial ^\alpha }f} \right|{\varepsilon ^{\left| \alpha  \right|}}}  \leq M{\varepsilon ^{k - \tau  - 1}}\varpi \left( \varepsilon  \right).
\end{equation}
Note that
\begin{align*}
	H_y^0\left( {x,0} \right) - \omega  ={}& H_y^0\left( {x,0} \right) - \sum\limits_{\left| \alpha  \right| \leq k - 1} {\partial _x^\alpha {H_y}\left( {\operatorname{Re} x,0} \right)\frac{{{{\left( {\mathrm{i}\operatorname{Im} x} \right)}^\alpha }}}{{\alpha !}}} \notag \\
	{}&+ \sum\limits_{\left| \alpha  \right| \leq k - 1} {{\partial ^\alpha }f\left( {\operatorname{Re} x} \right)\frac{{{{\left( {\mathrm{i}\operatorname{Im} x} \right)}^\alpha }}}{{\alpha !}}}.
\end{align*}
Therefore, for $ \left| {\operatorname{Im} x} \right| \leq \theta \varepsilon  $, by using \eqref{T1-3} and \eqref{pianf} we obtain that
\begin{align*}
	\left| {H_y^0\left( {x,0} \right) - \omega } \right| &\leq {c_1}{\left\| H \right\|_\varpi }{\varepsilon ^{k - 1}}\varpi \left( \varepsilon  \right) + M{\varepsilon ^{k - \tau  - 1}}\varpi \left( \varepsilon  \right)\notag \\
	&\leq c{\varepsilon ^{k - \tau  - 1}}\varpi \left( \varepsilon  \right) \notag \\
	&\leq c{\varepsilon ^{k - 2\tau  - 2}}\varpi \left( \varepsilon  \right) \cdot {\left( {\theta \varepsilon } \right)^{\tau  + 1}}.
\end{align*}
Recall \eqref{Q0}. Then it follows from \eqref{T1-3} that
\begin{align*}
	\left| {H_{yy}^0\left( z \right) - {Q^0}\left( z \right)} \right| &\leq {c_1}{\left\| H \right\|_\varpi }{\varepsilon ^{k - 2}}\varpi \left( \varepsilon  \right) \notag \\
	&\leq \frac{c}{{4M}}{\varepsilon ^{k - 2}}\varpi \left( \varepsilon  \right) \\
	&\leq \frac{c}{{4M}}{\varepsilon ^{k - 2\tau  - 2}}\varpi \left( \varepsilon  \right),\;\;\left| {\operatorname{Im} x} \right|,\left| y \right| \leq \theta \varepsilon,
\end{align*}
and
\begin{equation}\notag
	\left| {{Q^0}\left( z \right)} \right| \leq \sum\limits_{\left| \alpha  \right| \leq k - 2} {{{\left\| H \right\|}_\varpi }\frac{{{\varepsilon ^{\left| \alpha  \right|}}}}{{\alpha !}}}  \leq {\left\| H \right\|_\varpi }\sum\limits_{\alpha  \in {\mathbb{N}^{2n}}} {\frac{{{\varepsilon ^{\left| \alpha  \right|}}}}{{\alpha !}}}  = {\left\| H \right\|_\varpi }{e^{2n\varepsilon }} \leq 2M,\;\; \left| {\operatorname{Im} z} \right| \leq \varepsilon  .
\end{equation}

Now, by taking $ r^{*} = \theta \varepsilon ,\delta^{*}  = {\varepsilon ^{k - 2\tau  - 2}}\varpi \left( \varepsilon  \right) $ and using  Theorem \ref{appendix} there exists a real analytic symplectic transformation $ z = {\phi ^0}\left( \zeta  \right) $ of the form \eqref{bhxs} (with $ \nu=0 $) mapping the strip $ \left| {\operatorname{Im} \xi } \right|,\left| \eta  \right| \leq  {r_1}=r_0/2 $ into $ \left| {\operatorname{Im} x} \right|,\left| y \right| \leq \theta{r_0}=r_0/\sqrt{2} $,
such that $ {u^0}\left( \xi  \right) - \xi  $ and $ {v^0}\left( \xi  \right) $ are of period $ 1 $ in all variables and the Hamiltonian function $ {K^0}: = {H^0} \circ {\phi ^0} $  satisfies \eqref{qiudao} (with $ \nu=0 $). Moreover, \eqref{2.82}-\eqref{2.84}  (with $ \nu=0 $) hold.

Also assume that
\begin{equation}\notag
	\left| {K_{\eta \eta }^{\nu  - 1}\left( \zeta  \right)} \right| \leq {M_{\nu  - 1}},\;\;\left| {{{\left( {\int_{{\mathbb{T}^n}} {K_{\eta \eta }^{\nu  - 1}\left( {\xi ,0} \right){\rm d}\xi } } \right)}^{ - 1}}} \right| \leq {M_{\nu  - 1}},\;\;{M_\nu } \leq M
\end{equation}
for $ \left| {\operatorname{Im} x} \right| ,\left| y \right| \leq {r_\nu } $. Finally, define
\[\tilde H\left( {x,y} \right): = {H^\nu } \circ {\phi ^{\nu  - 1}}\left( {x,y} \right)\]
with respect to $ \left| {\operatorname{Im} x} \right| ,\left| y \right| \leq {r_\nu } $. One can verify that $ {\tilde H} $ is well defined.

Next we assume that the transformation $ z = {\phi ^{\nu  - 1}}\left( \zeta  \right) $ of the form \eqref{bhxs} has been constructed, mapping $ \left| {\operatorname{Im} \xi } \right|,\left| \eta  \right| \leq \theta {r_\nu } $ into $ \left| {\operatorname{Im} x} \right|,\left| {\operatorname{Im} y} \right| \leq {r_{\nu  - 1}},\left| {\operatorname{Re} y} \right| \leq \rho  $, and $ {u^{\nu  - 1}}\left( \xi  \right) - \xi ,{v^{\nu  - 1}}\left( \xi  \right) $ are of period $ 1 $ in all variables, and $ K_\xi ^{\nu  - 1}\left( {\xi ,0} \right) = 0,K_\eta ^{\nu  - 1}\left( {\xi ,0} \right) = \omega  $. In addition, we also assume that \eqref{2.82}-\eqref{2.85} hold for $ 0, \ldots ,\nu  - 1 $. In the next Step 3, we will verify that the above still hold for $ \nu $, which establishes a complete induction.\vspace{3mm}
\\
{\textbf{Step3:}} We will prove the existence of transformation $ {\phi ^\nu } $ in each step according to the specific estimates below and  Theorem \ref{appendix}.

Let $ \left| {\operatorname{Im} x} \right| \leq \theta {r_\nu } $. Then $ \phi^{\nu-1}(x,0) $ lies in the region where the estimates in \eqref{T1-3} hold for both $ H^\nu $ and $ H^{\nu-1} $. Note that $ x \mapsto H^{\nu-1}(\phi^{\nu-1}(x,0)) $ is constant by \eqref{qiudao}. Then by \eqref{T1-3}, we arrive at the following for $ \left| {\operatorname{Im} x} \right| \leq \theta {r_\nu } $
\begin{align*}
	\left| {\tilde H\left( {x,0} \right) - \int_{{\mathbb{T}^n}} {\tilde H\left( {\xi ,0} \right){\rm d}\xi } } \right| &\leq 2\mathop {\sup }\limits_{\left| {\operatorname{Im} \xi } \right| \leq \theta {r_\nu }} \left| {{H^\nu }\left( {{\phi ^{\nu  - 1}}\left( {\xi ,0} \right)} \right) - {H^{\nu  - 1}}\left( {{\phi ^{\nu  - 1}}\left( {\xi ,0} \right)} \right)} \right|\notag \\
	&\leq 2{c_1}{\left\| H \right\|_\varpi }r_\nu ^k\varpi \left( {{r_\nu }} \right) + 2{c_1}{\left\| H \right\|_\varpi }r_{\nu-1} ^{k}\varpi \left( {{r_{\nu-1} }} \right)\notag \\
	&\leq cr_\nu ^{k - 2\tau  - 2}\varpi \left( {{r_\nu }} \right) \cdot r_\nu ^{2\tau  + 2},
\end{align*}
where the weak homogeneity of $ \varpi $ with respect to $ a=1/2 $ (see Definition \ref{weak}) has been used in the last inequality, because $ \varpi(r_{\nu-1})=\varpi(2r_{\nu})\leq c \varpi(r_{\nu}) $ (thus $ c $ is independent of $ \nu $). For convenience we may therefore not mention it in the following.

Taking $ \eta=0 $ in \eqref{2.83} we have
\begin{align}
	\left| {u_\xi ^{\nu  - 1}\left( \xi  \right) - \mathbb{I}} \right| &\leq \sum\limits_{\mu  = 0}^{\nu  - 1} {\left| {u_\xi ^\mu \left( \xi  \right) - u_\xi ^{\mu  - 1}\left( \xi  \right)} \right|}  \notag \\
	&\leq c\sum\limits_{\mu  = 0}^{\nu  - 1} {r_\mu ^{k - 2\tau  - 2}\varpi \left( {{r_\mu }} \right)}  \notag \\
	&\leq c\sum\limits_{\mu  = 0}^\infty  {{{\left( {\frac{\varepsilon }{{{2^\mu }}}} \right)}^{k - 2\tau  - 2}}\varpi \left( {\frac{\varepsilon }{{{2^\mu }}}} \right)}  \notag \\
	&\leq c\sum\limits_{\mu  = 0}^\infty  {\left( {\frac{\varepsilon }{{{2^{\mu  - 1}}}} - \frac{\varepsilon }{{{2^\mu }}}} \right){{\left( {\frac{\varepsilon }{{{2^\mu }}}} \right)}^{k - 2\tau  - 3}}\varpi \left( {\frac{\varepsilon }{{{2^\mu }}}} \right)}  \notag \\
		&   \leq c\sum\limits_{\mu  = 0}^\infty  {\int_{\varepsilon /{2^\mu }}^{\varepsilon /{2^{\mu  - 1}}} {\frac{{\varpi \left( x \right)}}{{{x^{2\tau  + 3 - k}}}}{\rm d}x} }  \notag \\
	&\leq c\int_0^{2\varepsilon } {\frac{{\varpi \left( x \right)}}{{{x^{2\tau  + 3 - k}}}}{\rm d}x}  \notag \\
	\label{dao}&\leq 1 - \theta
\end{align}
for $ \left| {\operatorname{Im} \xi } \right| \leq \theta {r_\nu } $, and the Dini type condition \eqref{Dini} in (H1) together with Cauchy Theorem are used since $ \varepsilon>0 $ is sufficiently small. Then it leads to
\begin{equation}\label{nidao}
	\left| {u_\xi ^{\nu  - 1}{{\left( \xi  \right)}^{ - 1}}} \right| \leq {\theta ^{ - 1}},\;\;\left| {\operatorname{Im} \xi } \right| \leq \theta {r_\nu }.
\end{equation}

Finally, by \eqref{nidao} and \eqref{T1-3} we obtain that
\begin{align*}
	\left| {{{\tilde H}_y}\left( {x,0} \right) - \omega } \right| &= \left| {u_\xi ^{\nu  - 1}{{\left( x \right)}^{ - 1}}\left( {H_y^\nu \left( {{\phi ^{\nu  - 1}}\left( {x,0} \right)} \right) - H_y^{\nu  - 1}\left( {{\phi ^{\nu  - 1}}\left( {x,0} \right)} \right)} \right)} \right|\notag \\
	& \leq {\theta ^{ - 1}}\left| {H_y^\nu \left( {{\phi ^{\nu  - 1}}\left( {x,0} \right)} \right) - H_y^{\nu  - 1}\left( {{\phi ^{\nu  - 1}}\left( {x,0} \right)} \right)} \right|\notag \\
	& \leq {\theta ^{ - 1}}\left( {{c_1}{{\left\| H \right\|}_\varpi }r_\nu ^{k - 1}\varpi \left( {{r_\nu }} \right) + {c_1}{{\left\| H \right\|}_\varpi }r_{\nu  - 1}^{k - 1}\varpi \left( {{r_{\nu  - 1}}} \right)} \right)\notag \\
	&\leq cr_\nu ^{k - 1}\varpi \left( {{r_\nu }} \right)\notag \\
	&\leq cr_\nu ^{k - \tau  - 2}\varpi \left( {{r_\nu }} \right) \cdot r_\nu ^{\tau  + 1},
\end{align*}
and
\begin{align*}
	\left| {{{\tilde H}_{yy}}\left( z \right) - {Q^\nu }\left( z \right)} \right| &= \left| {u_\xi ^{\nu  - 1}{{\left( x \right)}^{ - 1}}\left( {H_{yy}^\nu \left( {{\phi ^{\nu  - 1}}\left( z \right)} \right) - H_{yy}^{\nu  - 1}\left( {{\phi ^{\nu  - 1}}\left( z \right)} \right)} \right){{\left( {u_\xi ^{\nu  - 1}{{\left( x \right)}^{ - 1}}} \right)}^{\top}}} \right|\notag \\
	&\leq {\theta ^{ - 2}}\left| {H_{yy}^\nu \left( {{\phi ^{\nu  - 1}}\left( z \right)} \right) - H_{yy}^{\nu  - 1}\left( {{\phi ^{\nu  - 1}}\left( z \right)} \right)} \right|\notag \\
	&\leq {\theta ^{ - 2}}\left( {{c_1}{{\left\| H \right\|}_\varpi }r_\nu ^{k - 2}\varpi \left( {{r_\nu }} \right) + {c_1}{{\left\| H \right\|}_\varpi }r_{\nu  - 1}^{k - 2}\varpi \left( {{r_{\nu  - 1}}} \right)} \right)\notag \\
	&\leq cr_\nu ^{k - 2\tau  - 2}\varpi \left( {{r_\nu }} \right)/2M
\end{align*}
for $ \left| {\operatorname{Im} x} \right|,\left| y \right| \leq \theta {r_\nu } $. Then denote $ r^*:= r_\nu $ and $ \delta^*:=c r_\nu ^{k - 2\tau  - 2}\varpi \left( {{r_\nu }} \right) $ in   Theorem \ref{appendix}, we obtain the analytic symplectic preserving transformation $ {\phi ^\nu } $ of each step, mapping the strip $ \left| {\operatorname{Im} \xi } \right|\leq \theta {r_\nu },\left| \eta  \right| \leq \theta {r_\nu } $ into $ \left| {\operatorname{Im} x} \right|\leq {r_\nu },\left| y \right| \leq {r_\nu } $, such that $ {u^\nu }\left( \xi  \right) - \xi $ and $ {v^\nu }\left( \xi  \right) $ are of period $ 1 $ in all variables, and the transformed Hamiltonian function $ {K^\nu } = {H^\nu } \circ {\phi ^\nu } $ satisfies
\[K_\xi ^\nu \left( {\xi ,0} \right) = 0,\;\;K_\eta ^\nu \left( {\xi ,0} \right) = \omega .\]
Moreover, \eqref{2.82}-\eqref{2.85} are valid for $ \left| {\operatorname{Im} \xi } \right|,\left| \eta  \right|,\left| {\operatorname{Im} x} \right|\leq \theta {r_\nu }$.\vspace{3mm}
\\
{\textbf{Step4:}} By  \eqref{2.83} for $ 0, \ldots ,\nu  - 1 $ and the arguments in \eqref{dao}, there holds
\begin{align}
	\left| {\phi _\zeta ^{\nu  - 1}\left( \zeta  \right)} \right| &\leq 1 + \sum\limits_{\mu  = 0}^{\nu  - 1} {\left| {\phi _\zeta ^\mu \left( \zeta  \right) - \phi _\zeta ^{\mu  - 1}\left( \zeta  \right)} \right|}  \notag \\
	&\leq 1 + \sum\limits_{\mu  = 0}^{\nu  - 1} {\left( {\left| {\phi _\zeta ^\mu \left( \zeta  \right) - \mathbb{I}} \right| + \left| {\phi _\zeta ^{\mu  - 1}\left( \zeta  \right) - \mathbb{I}} \right|} \right)} \notag \\
		&\leq 1 + c\sum\limits_{\mu  = 0}^\infty  {{{\left( {\frac{\varepsilon }{{{2^\mu }}}} \right)}^{k - 2\tau  - 2}}\varpi \left( {\frac{\varepsilon }{{{2^\mu }}}} \right)}   \notag \\
	&\leq 1 + c\int_0^{2\varepsilon } {\frac{{\varpi \left( x \right)}}{x^{2\tau +3-k}}{\rm d}x}   \notag \\
\label{dao2}	&\leq 2
\end{align}
for $ \left| {\operatorname{Im} \xi } \right|,\left| \eta  \right| \leq \theta {r_\nu } $ as long as $ \varepsilon>0 $ is sufficiently small, which leads to
\begin{align*}
	\left| {{\phi ^\nu }\left( \zeta  \right) - {\phi ^{\nu  - 1}}\left( \zeta  \right)} \right| &= \left| {{\phi ^{\nu  - 1}}\left( {{\psi ^\nu }\left( \zeta  \right)} \right) - {\phi ^{\nu  - 1}}\left( \zeta  \right)} \right| \notag \\
	&\leq 2\left| {{\psi ^\nu }\left( \zeta  \right) - \zeta } \right| \notag \\
	&\leq c\left( {1 - \theta } \right)r_\nu ^{k - 2\tau  - 1}\varpi \left( {{r_\nu }} \right)
\end{align*}
for $ \left| {\operatorname{Im} \xi } \right|,\left| \eta  \right| \leq {r_{\nu  + 1}} $. Then by Cauchy's estimate, we obtain that
\begin{equation}\notag
	\left| {\phi _\zeta ^\nu \left( \zeta  \right) - \phi _\zeta ^{\nu  - 1}\left( \zeta  \right)} \right| \leq cr_\nu ^{k - 2\tau  - 2}\varpi \left( {{r_\nu }} \right),\;\;\left| {\operatorname{Im} \xi } \right|,\left| \eta  \right| \leq {r_{\nu  + 1}}.
\end{equation}
It can be proved in the same way that $ | {\phi _\zeta ^\nu \left( \zeta  \right)} | \leq 2 $ for $ \left| {\operatorname{Im} \xi } \right|,\left| \eta  \right| \leq \theta {r_{\nu  + 1}} $, which implies
\begin{equation}\notag
	\left| {\operatorname{Im} z} \right| \leq 2\left| {\operatorname{Im} \zeta } \right| \leq 2\sqrt {{{\left| {\operatorname{Im} \xi } \right|}^2} + {{\left| {\operatorname{Im} \eta } \right|}^2}}  \leq 2\sqrt {{\theta ^2}r_{\nu  + 1}^2 + {\theta ^2}r_{\nu  + 1}^2}  = 2{r_{\nu  + 1}} = {r_\nu }.
\end{equation}
Besides, we have $ \left| {\operatorname{Re} y} \right| \leq \rho  $.

Note that
\begin{equation}\notag
	{v^\nu } \circ {\left( {{u^\nu }} \right)^{ - 1}}\left( x \right) - {v^{\nu  - 1}} \circ {\left( {{u^{\nu  - 1}}} \right)^{ - 1}}\left( x \right) = {\left( {u_\xi ^{\nu  - 1}{{\left( \xi  \right)}^{ - 1}}} \right)^{\top}}U_x^\nu \left( \xi  \right),\;\;x: = {u^{\nu  - 1}}\left( \xi  \right).
\end{equation}
Recall \eqref{dao}. By employing the contraction mapping principle we have $ \left| {\operatorname{Im} \xi } \right| \leq {r_{\nu  + 1}} $ if $ \left| {\operatorname{Im} x} \right| \leq \theta {r_{\nu  + 1}} $ with respect to $ x $ defined above. Then from \eqref{2.85} and \eqref{nidao} one can verify that
\begin{equation}\label{4.97}
	\left| {{{\big( {u_\xi ^{\nu  - 1}{{\left( \xi  \right)}^{ - 1}}} \big)}^{\top}}U_x^\nu \left( \xi  \right)} \right| \leq cr_\nu ^{k - \tau  - 1}\varpi \left( {{r_\nu }} \right).
\end{equation}
{\textbf{Step5:}} Finally, we are in a position  to  prove the uniform convergence of $ u^\nu $ and $ v^\nu $, and the regularity of their limit functions. Note \eqref{4.97}. Then we obtain the following analytic iterative scheme
\begin{equation}\label{4.98}
	\left| {{u^\nu }\left( \xi  \right) - {u^{\nu  - 1}}\left( \xi  \right)} \right| \leq cr_\nu ^{k - 2\tau  - 1}\varpi \left( {{r_\nu }} \right), \;\; \left| {\operatorname{Im} \xi } \right| \leq {r_{\nu  + 1}},
\end{equation}
and
\begin{equation}\label{4.99}
	\left| {{v^\nu } \circ {{\left( {{u^\nu }} \right)}^{ - 1}}\left( x \right) - {v^{\nu  - 1}} \circ {{\left( {{u^{\nu  - 1}}} \right)}^{ - 1}}\left( x \right)} \right| \leq c r_\nu ^{k - \tau  - 1}\varpi \left( {{r_\nu }} \right),\;\;\left| {\operatorname{Im} x} \right| \leq \theta {r_{\nu  + 1}}.
\end{equation}
And especially, \eqref{4.98} and \eqref{4.99} hold when $ \nu=0 $ since $ {u^{0 - 1}} = \mathrm{id} $ and $ {v^{ 0- 1}} = 0 $. It is obvious to see that the uniform limits $ u $ and $ v\circ u^{-1} $  of  $ u^\nu $ and $ v^\nu\circ (u^\nu)^{-1} $ are at least $ C^1 $ (in fact, this is implied by the higher regularity studied later in Section \ref{furtherregularity}, or one could derive the lower one using simpler technique than that in Theorem \ref{t1}, which we omit here). In addition, the persistent invariant torus possesses the same universal Diophantine (class $ \tau>n-1 $) frequency $ \omega $  as the unperturbed torus by \eqref{qiudao}.

\subsection{Iteration theorem on regularity beyond H\"older's type}\label{furtherregularity}
To obtain accurate regularity for $ u $ and $ v\circ u^{-1} $ from the analytic iterative scheme \eqref{4.98} and \eqref{4.99},  we shall along with the idea of Moser and Salamon to establish an \textit{abstract iterative theorem} on regularity beyond H\"older's type, which explicitly  shows the  modulus of continuity of the integral form. It should be emphasized that, from H\"older to non-H\"older,  the estimates here are much more difficult.
\begin{theorem}[Abstract iterative theorem]\label{t1}
	Let $ n\in \mathbb{N}^+, \varepsilon>0 $ and $ \{r_\nu\}_{\nu \in \mathbb
		N}=\{\varepsilon2^{-\nu}\}_{n \in \mathbb{N}} $ be given, and denote by $ f:{\mathbb{R}^n} \to \mathbb{R} $ the limit of a sequence of real analytic functions $ {f_\nu }\left( x \right) $ in the strips $ \left| {\operatorname{Im} x} \right| \leq {r_{\nu} } $ such that
	\begin{equation}\label{huoche}
		{f_0} = 0,\;\;\left| {{f_\nu }\left( x \right) - {f_{\nu  - 1}}\left( x \right)} \right| \leq \varphi \left( {{r_\nu }} \right),\;\;\nu  \geq 1,
	\end{equation}
	where $ \varphi $ is a nondecreasing continuous function satisfying $  \varphi \left( 0 \right) = 0  $.		Assume that there is a critical $ k_* \in \mathbb{N} $ such that
	\begin{equation}\label{330}
		\int_0^1 {\frac{{\varphi \left( x \right)}}{{{x^{{k_*} + 1}}}}{\rm d}x}  <  + \infty ,\;\;\int_0^1 {\frac{{\varphi \left( x \right)}}{{{x^{{k_*} + 2}}}}{\rm d}x}  =  + \infty .
	\end{equation}
	Then there exists a modulus of continuity $ \varpi_*  $  such that $ f \in {C_{k_*,\varpi_* }}\left( {{\mathbb{R}^n}} \right) $.  In other words, the regularity of $ f $ is at least of $ C^{k_*} $ plus $ \varpi_* $. In particular, $ \varpi_* $ could be determined as
	\begin{equation}\label{LLL}
		{\varpi _ * }\left( \gamma \right) \sim	\gamma \int_{L\left( \gamma  \right)}^\varepsilon  {\frac{{\varphi \left( t \right)}}{{{t^{{k_*} + 2}}}}{\rm d}t}  = {\mathcal{O}^\# }\left( {\int_0^{L\left( \gamma  \right)} {\frac{{\varphi \left( t \right)}}{{{t^{{k_*} + 1}}}}{\rm d}t} } \right) ,\;\;\gamma  \to {0^ + },
	\end{equation}
	where $ L(\gamma) \to 0^+ $ is some function such that the second relation in \eqref{LLL} holds.
\end{theorem}
\begin{proof}
	Define $ {g_\nu }(x): = {f_\nu }\left( x \right) - {f_{\nu  - 1}}\left( x \right) $ for all $ \nu \in \mathbb{N}^+ $.	Determine an integer function $ \widetilde N(\gamma) : [0,1] \to \mathbb{N}^+ $ (note that $ \widetilde N(\gamma) $ can be extended to $ \mathbb{R}^+ $ due to the arguments below, we thus assume that  it is  a continuous function). Then for the given critical  $ k_*\in \mathbb{N} $ and $ x,y\in \mathbb{R}^n $, we obtain  the following  for all multi-indices $ \alpha  = \left( {{\alpha _1}, \ldots ,{\alpha _n}} \right) \in {\mathbb{N}^n} $ with $ \left| \alpha  \right| = k_* $:
	\begin{align}
		\sum\limits_{\nu  = 1}^{\widetilde N\left( {\left| {x - y} \right|} \right)-1} {\left| {{\partial ^\alpha }{g_\nu }\left( x \right) - {\partial ^\alpha }{g_\nu }\left( y \right)} \right|}&\leq \left| {x - y} \right|\sum\limits_{\nu  = 1}^{\widetilde N\left( {\left| {x - y} \right|} \right)-1} {{{\left| {{\partial ^\alpha }{g_{\nu x}}} \right|}_{{C^0}(\mathbb{R}^n)}}}\notag \\
		&\leq \left| {x - y} \right|\sum\limits_{\nu  = 1}^{\widetilde N\left( {\left| {x - y} \right|} \right)-1} {\frac{1}{{r_\nu ^{k_* + 1}}}\varphi \left( {{r_\nu }} \right)} \notag \\
		& = 2\left| {x - y} \right|\sum\limits_{\nu  = 1}^{\widetilde N\left( {\left| {x - y} \right|} \right)-1} {\left( {\frac{\varepsilon }{{{2^\nu }}} - \frac{\varepsilon }{{{2^{\nu  + 1}}}}} \right){{\left( {\frac{{{2^\nu }}}{\varepsilon }} \right)}^{{k_ * } + 2}}\varphi \left( {\frac{\varepsilon }{{{2^\nu }}}} \right)}\notag  \\
		\label{zhengze1}&\leq  c\left| {x - y} \right|\int_{\varepsilon {2^{ - \widetilde N\left( {\left| {x - y} \right|} \right) }}}^\varepsilon  {\frac{{\varphi \left( t \right)}}{{{t^{{k_*} + 2}}}}{\rm d}t} ,
	\end{align}
	where Cauchy's estimate and \eqref{huoche} are used in the second inequality,  and arguments similar  to  \eqref{dao} are employed in \eqref{zhengze1}, $ c>0 $ is a universal constant. Besides, we similarly get
	\begin{align}
		\sum\limits_{\nu  = \widetilde N\left( {\left| {x - y} \right|} \right) }^\infty  {\left| {{\partial ^\alpha }{g_\nu }\left( x \right) - {\partial ^\alpha }{g_\nu }\left( y \right)} \right|} &\leq \sum\limits_{\nu  = \widetilde N\left( {\left| {x - y} \right|} \right) }^\infty  {2{{\left| {{\partial ^\alpha }{g_\nu }} \right|}_{{C^0}(\mathbb{R}^n)}}}\notag \\
		&\leq 2\sum\limits_{\nu  = \widetilde N\left( {\left| {x - y} \right|} \right)}^\infty  {\frac{1}{{r_\nu ^{k_*}}}\varphi \left( {{r_\nu }} \right)}\notag \\
		&=2\sum\limits_{\nu  = \widetilde N\left( {\left| {x - y} \right|} \right)}^\infty  {\left( {\frac{\varepsilon }{{{2^\nu }}} - \frac{\varepsilon }{{{2^{\nu  + 1}}}}} \right){{\left( {\frac{{{2^\nu }}}{\varepsilon }} \right)}^{{k_ * } + 1}}\varphi \left( {\frac{\varepsilon }{{{2^\nu }}}} \right)}\notag \\
		\label{zhengze2}& \leq c\int_0^{\varepsilon {2^{ - \widetilde N\left( {\left| {x - y} \right|} \right) }}} {\frac{{\varphi \left(t \right)}}{{{t^{{k_*} + 1}}}}{\rm d}t} .
	\end{align}
	Now let us choose $ \widetilde{N}(\gamma) \to +\infty $ as $ \gamma \to 0^+ $ such that
	\begin{equation}\label{varpi*}
		\gamma \int_{L\left( \gamma  \right)}^\varepsilon  {\frac{{\varphi \left( t \right)}}{{{t^{{k_*} + 2}}}}{\rm d}t}  = {\mathcal{O}^\# }\left( {\int_0^{L\left( \gamma  \right)} {\frac{{\varphi \left( t \right)}}{{{t^{{k_*} + 1}}}}{\rm d}t} } \right): = {\varpi _ * }\left( \gamma \right),\;\;\gamma  \to {0^ + },
	\end{equation}
	where $ \varepsilon {2^{ - \widetilde N\left( \gamma  \right) - 1}}: = L\left( \gamma  \right)\to 0^+ $. This is achievable due to assumption \eqref{330}, Cauchy Theorem  and The Intermediate Value Theorem. Note that the choice of $ L(\gamma) $ (i.e., $ \widetilde{N} $) and $ \varpi_* $ is not unique (may up to a constant), and $ \varpi_* $ could be continuously extended to some given interval (e.g., $ [0,1] $), but this does not affect the qualitative result. Combining \eqref{zhengze1}, \eqref{zhengze2} and \eqref{varpi*} we finally arrive at $ f \in {C_{k_*,\varpi_* }}\left( {{\mathbb{R}^n}} \right) $ because
	\begin{align*}
		\left| {{\partial ^\alpha }f\left( x \right) - {\partial ^\alpha }f\left( y \right)} \right| &\leq \sum\limits_{\nu  = 1}^{\widetilde N\left( {\left| {x - y} \right|} \right)} { + \sum\limits_{\nu  = \widetilde N\left( {\left| {x - y} \right|} \right) + 1}^\infty  {\left| {{\partial ^\alpha }{g_\nu }\left( x \right) - {\partial ^\alpha }{g_\nu }\left( y \right)} \right|} } \\
		& \leq c\left( {\left| {x - y} \right|\int_{\varepsilon {2^{ - \widetilde N\left( {\left| {x - y} \right|} \right) - 1}}}^\varepsilon  {\frac{{\varphi \left( t \right)}}{{{t^{{k_*} + 2}}}}{\rm d}t}  + \int_0^{\varepsilon {2^{ - \widetilde N\left( {\left| {x - y} \right|} \right) - 1}}} {\frac{{\varphi \left( t \right)}}{{{t^{{k_*} + 1}}}}{\rm d}t} } \right) \\
		&\leq c\varpi_* \left( {\left| {x - y} \right|} \right).
	\end{align*}
	
\end{proof}

Theorem \ref{t1} can be extended to the case $ f:{\mathbb{R}^n} \to \mathbb{R}^m $ with $ n,m \in \mathbb{N}^+ $ since the analysis is completely  same,  and the strip $ \left| {\operatorname{Im} x} \right| \leq {r_\nu } $ can also be replaced by $ \left| {\operatorname{Im} x} \right| \leq {r_{\nu+1} } $ (or $ \leq \theta r_{\nu+1} $).  Theorem \ref{t1} can also be used to estimate the regularity of solutions of finite smooth homological equations, thus KAM uniqueness theorems in some cases might be derived, see Section 4 in \cite{salamon} for instance. However, in order to avoid too much content in this paper, it is omitted here.

Finally, by recalling \eqref{4.98} and \eqref{4.99}, one applies  Theorem \ref{t1} on $ \{u^{\nu}-\mathrm{id}\}_\nu $ (because Theorem \ref{t1} requires that the initial value vanishes) and $ \{v^{\nu}\circ (u^\nu)^{-1}\}_\nu $ to directly analyze the regularity of  KAM torus and conjugation according to (H5), i.e., there exist modulus of continuity $ {\varpi _i} $ ($ i=1,2 $) such that $ u \in {C_{k_1^ * ,{\varpi _1}}}\left( {{\mathbb{R}^n},{\mathbb{R}^n}} \right) $ and $ v \circ {u^{ - 1}} \in {C_{k_2^ * ,{\varpi _2}}}\left( {{\mathbb{R}^n},G} \right) $. This completes the proof of Theorem \ref{theorem1}.

\section{Proof of other results}\label{ProofOther}

\subsection{Proof of  Theorem \ref{Holder}}\label{8}

Note that $ \ell  \notin {\mathbb{N}^ + } $ implies $ \{\ell\}\in (0,1) $. Then $ k=[\ell] $ and $ \varpi(x)\sim \varpi_{\mathrm{H}}^{\ell}(x)\sim x^{\{\ell\}} $, i.e., modulus of continuity of H\"older's type. Consequently, (H1) can be directly verified because of  $ \ell  > 2\tau  + 2 $:
\begin{align*}
	\int_0^1 {\frac{{\varpi \left( x \right)}}{{{x^{2\tau  + 3 - k}}}}{\rm d}x} & = \int_0^1 {\frac{{{x^{\left\{ \ell  \right\}}}}}{{{x^{2\tau  + 3 - \left[ \ell  \right]}}}}{\rm d}x}  \notag \\
	&= \int_0^1 {\frac{1}{{{x^{1 - \left( {\ell  - 2\tau  - 2} \right)}}}}{\rm d}x}  <  + \infty .
\end{align*}
Here and below, let $ i $ be $ 1 $ or $ 2 $ for simplicity. Recall that $ {\varphi _i}\left( x \right) = {x^{k - \left( {3 - i} \right)\tau  - 1}}\varpi \left( x \right) = {x^{\left[ \ell  \right] - \left( {3 - i} \right)\tau  - 1}} \times {x^{\left\{ \ell  \right\}}} = {x^{\ell  - \left( {3 - i} \right)\tau  - 1}} $, and let
\begin{align}
	\label{fff}\int_0^1 {\frac{{{\varphi _i}\left( x \right)}}{{{x^{k_i^ *  + 1}}}}{\rm d}x}  &= \int_0^1 {\frac{1}{{{x^{k_i^ *  - \left( {\ell  - \left( {3 - i} \right)\tau  - 2} \right)}}}}{\rm d}x}<+\infty,\\
	\label{ffff}\int_0^1 {\frac{{{\varphi _i}\left( x \right)}}{{{x^{k_i^ *  + 2}}}}{\rm d}x}  &= \int_0^1 {\frac{1}{{{x^{k_i^ *  - \left( {\ell  - \left( {3 - i} \right)\tau  - 2} \right) + 1}}}}{\rm d}x}=+\infty .
\end{align}
Then the critical $ k_i^* $ in (H5) could be uniquely chosen as $ k_i^ * : = \left[ {\ell  - \left( {3 - i} \right)\tau  - 1} \right] \in \mathbb{N}^+$ since $ \ell  - \left( {3 - i} \right)\tau  - 1 \notin \mathbb{N}^+ $. Further, letting $ {L_i}\left( \gamma  \right) = \gamma  \to {0^ + } $ yields 
\begin{align*}
	\int_0^{{L_i}\left( \gamma  \right)} {\frac{{{\varphi _i}\left( t \right)}}{{{t^{k_i^ *  + 1}}}}{\rm d}t}  &= {\mathcal{O}^\# }\left( {\int_0^\gamma  {\frac{1}{{{t^{1 - \left\{ {\ell  - \left( {3 - i} \right)\tau  - 2} \right\}}}}}{\rm d}t} } \right) \notag \\
	&= {\mathcal{O}^\# }\left( {{\gamma ^{\left\{ {\ell  - \left( {3 - i} \right)\tau  - 2} \right\}}}} \right)
\end{align*}
and
\begin{align*}
	\gamma \int_{{L_i}\left( \gamma  \right)}^\varepsilon  {\frac{{{\varphi _i}\left( t \right)}}{{{t^{k_i^ *  + 2}}}}{\rm d}t}  &= {\mathcal{O}^\# }\left( {\gamma \int_\gamma ^\varepsilon  {\frac{1}{{{t^{2 - \left\{ {\ell  - \left( {3 - i} \right)\tau  - 2} \right\}}}}}{\rm d}t} } \right) \notag \\
	&= {\mathcal{O}^\# }\left( {{\gamma ^{\left\{ {\ell  - \left( {3 - i} \right)\tau  - 2} \right\}}}} \right).
\end{align*}
This leads to modulus of continuity being of H\"older's type
\[{\varpi _i}\left( \gamma  \right) \sim {\left( {{L_i}\left( \gamma  \right)} \right)^{\left\{ {\ell  - \left( {3 - i} \right)\tau  - 2} \right\}}} \sim {\gamma ^{\left\{ {\ell  - \left( {3 - i} \right)\tau  - 2} \right\}}}\sim \varpi_{\mathrm{H}}^{\left\{ {\ell  - \left( {3 - i} \right)\tau  - 2} \right\}}(\gamma)\]
due to \eqref{varpii} in Theorem \ref{Theorem1}. 

Finally by observing the H\"older indices $ k_i^ *  + \left\{ {\ell  - \left( {3 - i} \right)\tau  - 2} \right\} = \ell  - \left( {3 - i} \right)\tau  - 1 $ for $ i=1,2 $, we  conclude that $ u \in {C^{\ell-2\tau-1}}\left( {{\mathbb{R}^n},{\mathbb{R}^n}} \right) $ and $ v \circ {u^{ - 1}} \in {C^{\ell-\tau-1}}\left( {{\mathbb{R}^n},G} \right) $. This proves Theorem \ref{Holder}.

\subsection{Proof of Theorem \ref{lognew}}

Firstly, note that $ k=[2\tau+2] $ and $ \varpi \left( x \right) \sim {x^{\{2\tau+2\}}}/{\left( { - \ln x} \right)^\lambda } $ with $ \lambda  > 1 $, then (H1) holds since
\begin{align*}
	\int_0^1 {\frac{{\varpi \left( x \right)}}{{{x^{2\tau  + 3 - k}}}}{\rm d}x}  &= {\mathcal{O}^\# }\left( {\int_0^{1/2} {\frac{{{x^{\left\{ {2\tau  + 2} \right\}}}}}{{{x^{2\tau  + 3 - \left[ {2\tau  + 2} \right]}}{{\left( { - \ln x} \right)}^\lambda }}}{\rm d}x} } \right)\\
	& = {\mathcal{O}^\# }\left( {\int_0^{1/2} {\frac{1}{{x{{\left( { - \ln x} \right)}^\lambda }}}{\rm d}x} } \right) <  + \infty .
\end{align*}
Thus  the frequency-preserving KAM persistence is obtained. The followings are devoted to the remaining regularity.

Secondly, in view of $ \varphi_i(x) $ in (H5), we have
\[\int_0^1 {\frac{{{\varphi _i}\left( x \right)}}{{{x^{k_i^ *  + 1}}}}{\rm d}x}  = {\mathcal{O}^\# }\left( {\int_0^{1/2} {\frac{1}{{{x^{k_i^ *  - \left( {i - 1} \right)\tau }}{{\left( { - \ln x} \right)}^\lambda }}}{\rm d}x} } \right)\;\;i=1,2.\]
This leads to critical $ k_1^*=1 $ and $ k_2^*=[\tau+1] $ in (H5). Here one uses the following fact: for given $ \lambda>1 $,
\[\int_0^{1/2} {\frac{1}{{{x^\iota }{{\left( { - \ln x} \right)}^\lambda }}}{\rm d}x}  <  + \infty ,\;\;\int_0^{1/2} {\frac{1}{{{x^{\iota  + 1}}{{\left( { - \ln x} \right)}^\lambda }}}{\rm d}x}=+\infty \]
if and only if $ \iota \in (0,1] $.

Next, we shall investigate the  KAM remaining regularity through certain complicated asymptotic analysis, that is, show the modulus of continuity $ \varpi_1 $ and $ \varpi_2 $ explicitly. We have to discuss this problem in two cases. This is because $ (i-1)\tau=0 $ when $ i= 1$ (corresponds to the conjugation), i.e., independent of the Diophantine property; while $ (i-1)\tau=\tau $ when $ i=2 $ (corresponds to the KAM torus). As can be seen later, they lead to different asymptotic analysis.\vspace{3mm}
\\
{\textbf{Case 1:}} Here we provide the analysis for $ \varpi_1 $ with all $ \tau>n-1 $, as well as $ \varpi_2 $ with $ n-1<\tau \in \mathbb{N}^+ $. In view of $ \varphi_i(x) $  in (H5), then by applying Lemma  \ref{duochongduishu} with $ \varrho=1 $ we get
\begin{align}
	\gamma \int_{{L_i}\left( \gamma  \right)}^\varepsilon  {\frac{{{\varphi _i}\left( t \right)}}{{{t^{k_i^ *  + 2}}}}{\rm d}t}  &= {\mathcal{O}^\# }\Bigg( {\gamma \int_{{L_i}\left( \gamma  \right)}^\varepsilon  {\frac{1}{{{t^2}(\ln (1/t))^\lambda  }}{\rm d}t} } \Bigg)\notag \\
	& = {\mathcal{O}^\# }\Bigg( {\gamma \int_{1/\varepsilon }^{1/{L_i}\left( \gamma  \right)} {\frac{1}{{(\ln z)^\lambda }}{\rm d}z} } \Bigg)\notag \\
	\label{Cguji1}& = {\mathcal{O}^\# }\Bigg( {\frac{\gamma }{{{L_i}\left( \gamma  \right)(\ln (1/{L_i}\left( \gamma  \right)))^\lambda}}} \Bigg),
\end{align}
and by direct calculation one arrives at
\begin{align}
	\int_0^{{L_i}\left( \gamma  \right)} {\frac{{{\varphi _i}\left( t \right)}}{{{t^{k_i^ *  + 1}}}}{\rm d}t} & = {\mathcal{O}^\# }\Bigg( {\int_{{L_i}\left( \gamma  \right)}^\varepsilon  {\frac{1}{{t(\ln (1/t))^\lambda }}{\rm d}t} } \Bigg)\notag \\
	\label{Cguji2}&= {\mathcal{O}^\# }\Bigg( {\frac{1}{{{{( {\ln  } (1/{L_i}\left( \gamma  \right)))}^{\lambda  - 1}}}}} \Bigg).
\end{align}
Finally, choosing
\begin{equation}\label{CLit4}
	{L_i}\left( \gamma  \right) \sim \frac{{\gamma }}{{\ln (1/\gamma) }} \to {0^ + },\;\;\gamma  \to {0^ + }
\end{equation}
leads to the second relation in \eqref{varpii} for $ i=1,2 $, and substituting $ L_i(\gamma) $ into \eqref{Cguji1} or \eqref{Cguji2} yields 
\[{\varpi _1}\left( \gamma \right) \sim \frac{1}{{{{\left( { - \ln \gamma} \right)}^{\lambda  - 1}}}} \sim \varpi _{\mathrm{LH}}^{\lambda  - 1}\left( \gamma \right), \;\; \tau >n-1\]
and
\[{\varpi _2}\left( \gamma \right) \sim \frac{1}{{{{\left( { - \ln \gamma} \right)}^{\lambda  - 1}}}} \sim \varpi _{\mathrm{LH}}^{\lambda  - 1}\left( \gamma \right), \;\; n-1<\tau \notin \mathbb{N}^+\]
in Theorem \ref{theorem1}, see \eqref{varpii}. \vspace{3mm}
\\
{\textbf{Case 2:}} 
However, the asymptotic analysis for $ \varpi_2 $ becomes more different when $ n-1<\tau \notin \mathbb{N}^+  $. Note that $ \{\tau\} \in (0,1) $ and $ \left[ {\tau  + 1} \right] - \tau  = \left[ \tau  \right] + 1 - \tau  = 1 - \left\{ \tau  \right\} $ at present. Hence, by  applying \eqref{erheyi1} in Lemma \ref{erheyi} we get
\begin{align}
	\int_0^{{L_2}\left( \gamma  \right)} {\frac{{{\varphi _2}\left( t \right)}}{{{t^{k_2^ *  + 1}}}}{\rm d}t}  &= {\mathcal{O}^\# }\left( {\int_0^{{L_2}\left( \gamma  \right)} {\frac{1}{{{t^{\left[ {\tau  + 1} \right] - \tau }}{{\left( { - \ln t} \right)}^\lambda }}}{\rm d}t} } \right) \notag \\
	&= {\mathcal{O}^\# }\left( {\int_0^{{L_2}\left( \gamma  \right)} {\frac{1}{{{t^{1 - \left\{ \tau  \right\}}}{{\left( { - \ln t} \right)}^\lambda }}}{\rm d}t} } \right)\notag \\
&= {\mathcal{O}^\# }\left( {\int_{1/{L_2}\left( \gamma  \right)}^{ + \infty } {\frac{1}{{{z^{1 + \left\{ \tau  \right\}}}{{\left( {\ln z} \right)}^\lambda }}}{\rm d}z} } \right) \notag \\
	\label{coro42}& 	= {\mathcal{O}^\# }\left( {\frac{{{{\left( {{L_2}\left( \gamma  \right)} \right)}^{\left\{ \tau  \right\}}}}}{{{{\left( {\ln \left( {1/{L_2}\left( \gamma  \right)} \right)} \right)}^\lambda }}}} \right),
\end{align}
and similarly according to \eqref{erheyi2} in Lemma \ref{erheyi} we have
\begin{align}
		\gamma \int_{{L_2}\left( \gamma  \right)}^\varepsilon  {\frac{{{\varphi _2}\left( t \right)}}{{{t^{k_2^ *  + 2}}}}{\rm d}t}  &= {\mathcal{O}^\# }\left( {\gamma \int_{1/\varepsilon }^{1/{L_2}\left( \gamma  \right)} {\frac{1}{{{z^{\left\{ \tau  \right\}}}{{\left( {\ln z} \right)}^\lambda }}}{\rm d}z} } \right)\notag \\
	\label{coro43}& = {\mathcal{O}^\# }\left( {\frac{{\gamma {{\left( {{L_2}\left( \gamma  \right)} \right)}^{\left\{ \tau  \right\} - 1}}}}{{{{\left( {\ln \left( {1/{L_2}\left( \gamma  \right)} \right)} \right)}^\lambda }}}} \right).	
\end{align}
Now let us choose $ L_2(\gamma) \sim \gamma \to 0^+ $, i.e., different from that when $ n-1<\tau \in \mathbb{N}^+ $ in Case 1, 	one verifies that the second relation in \eqref{varpii} holds for $ i=2 $, and substituting $ L_2(\gamma) $ into \eqref{coro42} or \eqref{coro43} yields the following modulus of continuity
\[{\varpi _2}\left( \gamma  \right) \sim \frac{{{\gamma ^{\left\{ \tau  \right\}}}}}{{{{\left( { - \ln \gamma } \right)}^\lambda }}} \sim {\gamma ^{\left\{ \tau  \right\}}}\varpi _{\mathrm{LH}}^\lambda \left( \gamma  \right),\;\; n-1<\tau \notin \mathbb{N}^+\]
due to \eqref{varpii} in Theorem \ref{theorem1}.
 
  This finishes the proof of Theorem \ref{lognew}.

\subsection{Proof of Theorem \ref{GLHnew}}
It is sufficient to determine $ k_i^* $ in (H5) and to choose functions $ L_i(\gamma)\to 0^+ $ (as $ \gamma \to 0^+ $),  obtaining the modulus of continuity $ \varpi_i $ in \eqref{varpii} for $ i=1,2 $. Obviously $ k_1^*=1 $ and $ k_2^*=\tau+1 $ due to $ \tau \in \mathbb{N}^+ $ and 
\[\int_0^1 {\frac{{\varpi _{\mathrm{GLH}}^{\varrho ,\lambda }\left( x \right)}}{x}{\rm d}x}  <  + \infty ,\;\;\int_0^1 {\frac{{\varpi _{\mathrm{GLH}}^{\varrho ,\lambda }\left( x \right)}}{{{x^2}}}{\rm d}x}  =  + \infty .\]
In view of $ \varphi_i(x) $  in (H5), then by applying Lemma  \ref{duochongduishu} we get
\begin{align}
	\gamma \int_{{L_i}\left( \gamma  \right)}^\varepsilon  {\frac{{{\varphi _i}\left( t \right)}}{{{t^{k_i^ *  + 2}}}}{\rm d}t}  &= {\mathcal{O}^\# }\Bigg( {\gamma \int_{{L_i}\left( \gamma  \right)}^\varepsilon  {\frac{1}{{{t^2}(\ln (1/t)) \cdots {{(\underbrace {\ln  \cdots \ln }_\varrho (1/t))}^\lambda }}}{\rm d}t} } \Bigg)\notag \\
	\label{guji1}& = {\mathcal{O}^\# }\Bigg( {\frac{\gamma }{{{L_i}\left( \gamma  \right)(\ln (1/{L_i}\left( \gamma  \right))) \cdots {{(\underbrace {\ln  \cdots \ln }_\varrho (1/{L_i}\left( \gamma  \right)))}^\lambda }}}} \Bigg),
\end{align}
and by direct calculation one derives
\begin{align}
	\int_0^{{L_i}\left( \gamma  \right)} {\frac{{{\varphi _i}\left( t \right)}}{{{t^{k_i^ *  + 1}}}}{\rm d}t} & = {\mathcal{O}^\# }\Bigg( {\int_{{L_i}\left( \gamma  \right)}^\varepsilon  {\frac{1}{{t(\ln (1/t)) \cdots {{(\underbrace {\ln  \cdots \ln }_\varrho (1/t))}^\lambda }}}{\rm d}t} } \Bigg)\notag \\
	\label{guji2}&= {\mathcal{O}^\# }\Bigg( {\frac{1}{{{{(\underbrace {\ln  \cdots \ln }_\varrho (1/{L_i}\left( \gamma  \right)))}^{\lambda  - 1}}}}} \Bigg).
\end{align}
Finally, choosing
\begin{equation}\label{Lit4}
	{L_i}\left( \gamma  \right) \sim \frac{{\gamma }}{{(\ln (1/\gamma)) \cdots (\underbrace {\ln  \cdots \ln }_{\varrho}(1/\gamma )})} \to {0^ + },\;\;\gamma  \to {0^ + }
\end{equation}
leads to the second relation in \eqref{varpii} for $ i=1,2 $, and substituting $ L_i(\gamma) $ into \eqref{guji1} or \eqref{guji2} yields that
\begin{equation}\label{rho1}
	{\varpi _1}\left( \gamma \right) \sim {\varpi _2}\left( \gamma \right) \sim \frac{1}{{{{(\underbrace {\ln  \cdots \ln }_\varrho (1/\gamma))}^{\lambda  - 1}}}}
\end{equation}
in Theorem \ref{theorem1}, see \eqref{varpii}, which proves Theorem \ref{GLHnew}.

\section{Appendix}

\subsection{Semi separability and weak homogeneity for modulus of continuity}
\begin{lemma}\label{Oxlemma}
	Let a modulus continuity $ \varpi $ be given. If $ \varpi $ is piecewise continuously differentiable and $ \varpi'\geq 0 $ is nonincreasing, then $ \varpi $ admits semi separability in Definition \ref{d2}. As a consequence, if $ \varpi $ is convex near $ 0^+ $, then it is semi separable.
\end{lemma}
\begin{proof}
	Assume that $ \varpi $ is continuously differentiable without loss of generality. Then we obtain semi separability due to
	\begin{align*}
		\mathop {\sup }\limits_{0 < r < \delta /x} \frac{{\varpi \left( {rx} \right)}}{{\varpi \left( r \right)}} &= \mathop {\sup }\limits_{0 < r < \delta /x} \frac{{\varpi \left( {rx} \right) - \varpi \left( 0+ \right)}}{{\varpi \left( r \right)}} \leq \mathop {\sup }\limits_{0 < r < \delta /x} \frac{1}{{\varpi \left( r \right)}}\sum\limits_{j = 0}^{\left[ x \right]} {\int_{jr}^{\left( {j + 1} \right)r} {\varpi '\left( t \right){\rm d}t} } \\
		& \leq \mathop {\sup }\limits_{0 < r < \delta /x} \frac{1}{{\varpi \left( r \right)}}\sum\limits_{j = 0}^{\left[ x \right]} {\int_0^r {\varpi '\left( t \right){\rm d}t} }  = \left( {\left[ x \right] + 1} \right) = \mathcal{O}\left( x \right),\;\;x \to  + \infty .
	\end{align*}
\end{proof}
\begin{lemma}\label{ruotux}
	Let a modulus continuity $ \varpi $ be given. If $ \varpi $ is convex near $ 0^+ $, then it admits weak homogeneity in Definition \ref{weak}.
\end{lemma}
\begin{proof}
	For $ x>0 $ sufficiently small, one verifies that
	\[\varpi \left( x \right) = x \cdot \frac{{\varpi \left( x \right) - \varpi \left( 0+ \right)}}{x-0} \leq x \cdot \frac{{\varpi \left( {ax} \right) - \varpi \left( 0+ \right)}}{{ax}-0} = a^{-1}\varpi \left( {ax} \right),\]
	for  $ 0<a<1 $, which leads to weak homogeneity
	\[	\mathop {\overline {\lim } }\limits_{x \to {0^ + }} \frac{{\varpi \left( x \right)}}{{\varpi \left( {ax} \right)}} \leq a^{-1} <  + \infty .\]
\end{proof}

\subsection{Proof of Theorem \ref{Theorem1}} \label{JACK}
\begin{proof}
	For the completeness of the analysis we shall give a very detailed proof. An outline of the strategy for the proof is provided: we firstly construct an approximation integral operator by the Fourier transform of  a compactly supported function, and then present certain  properties of the operator (note that these preparations are classical); finally, we estimate the approximation error in the sense of modulus of continuity.	
	
	Let\[K\left( x \right) = \frac{1}{{{{\left( {2\pi } \right)}^n}}}\int_{{\mathbb{R}^n}} {\widehat K\left( \xi \right){e^{\mathrm{i}\left\langle {x,\xi } \right\rangle }}{\rm d}\xi } ,\;\;x \in {\mathbb{C}^n}\]
	be an entire function, whose Fourier transform
	\[\widehat K\left( \xi  \right) = \int_{{\mathbb{R}^n}} {K\left( x \right){e^{ - \mathrm{i}\left\langle {x,\xi } \right\rangle }}{\rm d}x} ,\;\;\xi  \in {\mathbb{R}^n}\]
	is a smooth function with compact support, contained in the ball $ \left| \xi  \right| \leq 1 $, that satisfies $ \widehat K\left( \xi  \right) = \widehat K\left( { - \xi } \right) $ and
	\begin{equation}\label{3.3}
		{\partial ^\alpha }\widehat K\left( 0 \right) = \left\{ \begin{gathered}
			1,\;\;\alpha  = 0, \hfill \\
			0,\;\;\alpha  \ne 0. \hfill \\
		\end{gathered}  \right.
	\end{equation}
	Next, we assert that
	\begin{equation}\label{5}
		\left| {{\partial ^\beta }\mathcal{F}\big( {\widehat K\left( \xi  \right)} \big)\left( z \right)} \right|  \leq \frac{{c\left( {\beta ,p} \right)}}{{{{\left( {1 + \left| {\operatorname{Re} z} \right|} \right)}^p}}}{e^{\left| {\operatorname{Im} z} \right|}},\;\;\max \left\{ {1,\left|\beta\right|} \right\} \leq p \in \mathbb{R}.
	\end{equation}
	
	Note that we assume $ \widehat K \in C_0^\infty \left( {{\mathbb{R}^n}} \right) $ and $ \operatorname{supp} \widehat K \subseteq B\left( {0,1} \right) $, thus
	\begin{equation}\label{FK}
		\left| {{{\left( {1 + \left| z \right|} \right)}^k}{\partial ^\beta }\mathcal{F}\left( {\widehat K\left( \xi  \right)} \right)\left( z \right)} \right| \leq \sum\limits_{\left| \gamma  \right| \leq k} {\left| {{z^\gamma }{\partial ^\beta }\mathcal{F}\big( {\widehat K\left( \xi  \right)} \big)\left( z \right)} \right|}  = \sum\limits_{\left| \gamma  \right| \leq k} {\left| {{\partial ^{\beta  + \gamma }}\mathcal{F}\big( {\widehat K\left( \xi  \right)} \big)\left( z \right)} \right|},
	\end{equation}
	where $ \mathcal{F} $ represents the Fourier transform. Since $ {\partial ^{\beta  + \gamma }}\mathcal{F}\big( {\widehat K\left( \xi  \right)} \big)\left( z \right) \in C_0^\infty ( {\overline {B \left( {0,1} \right)}} ) $ does not change the condition, we only need to prove that
	\[\left| {\mathcal{F}\big( {\widehat K\left( \xi  \right)} \big)\left( z \right)} \right| \leq {c_k}{e^{\left| {\operatorname{Im} z} \right|}}.\]
	Obviously
	\[\left| {\mathcal{F}\big( {\widehat K\left( \xi  \right)} \big)\left( z \right)} \right| \leq \frac{1}{{{{\left( {2\pi } \right)}^n}}}\int_{{\mathbb{R}^n}} {\big| {\widehat K\left( \xi  \right)} \big|{e^{ - \left\langle {\operatorname{Im} z,\xi } \right\rangle }}{\rm d}\xi }  \leq \frac{c}{{{{\left( {2\pi } \right)}^n}}}\int_{B\left( {0,1} \right)} {{e^{\left| {\left\langle {\operatorname{Im} z,\xi } \right\rangle } \right|}}{\rm d}\xi }  \leq c{e^{\left| {\operatorname{Im} z} \right|}},\]
	where $ c>0 $ is independent of $ n $. Then assertion \eqref{5} is proved by recalling  \eqref{FK}.
	
	The inequality in \eqref{5} is usually called the Paley-Wiener Theorem, see also Chapter III in \cite{Stein}.	As we will see later, it plays an important role in the subsequent verification of definitions, integration by parts and the translational feasibility according to Cauchy's integral formula.
	
	Next we assert that $ K:{\mathbb{C}^n} \to \mathbb{R} $ is a real analytic function with the following property
	\begin{equation}\label{3.6}
		\int_{{\mathbb{R}^n}} {{{\left( {u + \mathrm{i}v} \right)}^\alpha }{\partial ^\beta }K\left( {u + \mathrm{i}v} \right){\rm d}u}  = \left\{ \begin{aligned}
			&{\left( { - 1} \right)^{\left| \alpha  \right|}}\alpha !,&\alpha  = \beta , \hfill \\
			&0,&\alpha  \ne \beta , \hfill \\
		\end{aligned}  \right.
	\end{equation}
	for $ u,v \in {\mathbb{R}^n} $ and multi-indices $ \alpha ,\beta  \in {\mathbb{N}^n} $. In order to prove assertion  \eqref{3.6}, we first consider proving the following for $ x\in \mathbb{R}^n $:
	\begin{equation}\label{3.7}
		\int_{{\mathbb{R}^n}} {{x^\alpha }{\partial ^\beta }K\left( x \right){\rm d}x}  = \left\{ \begin{aligned}
			&{\left( { - 1} \right)^{\left| \alpha  \right|}}\alpha !,&\alpha  = \beta  \hfill, \\
			&0,&\alpha  \ne \beta  \hfill. \\
		\end{aligned}  \right.
	\end{equation}
	{\bf{Case1:}} If $ \alpha  = \beta  $, then
	\begin{align*}
		\int_{{\mathbb{R}^n}} {{x^\alpha }{\partial ^\beta }K\left( x \right){\rm d}x}  &= \int_{{\mathbb{R}^n}} {\Big( {\prod\limits_{j = 1}^n {x_j^{{\alpha _j}}} } \Big)\Big( {\prod\limits_{j = 1}^n {\partial _{{x_j}}^{{\alpha _j}}} } \Big)K\left( x \right){\rm d}x} \\
		&= {\left( { - 1} \right)^{{\alpha _1}}}{\alpha _1}!\int_{{\mathbb{R}^{n - 1}}} {\Big( {\prod\limits_{j = 2}^n {x_j^{{\alpha _j}}} } \Big)\Big( {\prod\limits_{j = 2}^n {\partial _{{x_j}}^{{\alpha _j}}} } \Big)K\left( x_2,\cdots,x_n \right){\rm d}x_2 \cdots {\rm d}x_n} \\
		&=  \cdots  = {\left( { - 1} \right)^{{\alpha _1} +  \cdots  + {\alpha _n}}}{\alpha _1}! \cdots {\alpha _n}!\int_\mathbb{R} {K\left( x_n \right){\rm d}x_n} \\
		& = {\left( { - 1} \right)^{\left| \alpha  \right|}}\alpha !\widehat K\left( 0 \right) = {\left( { - 1} \right)^{\left| \alpha  \right|}}\alpha !.
	\end{align*}
	{\bf{Case2:}} There exists some $ {\alpha _j} \leq {\beta _j} - 1 $, let $ j=1 $ without loss of generality. Then
	\begin{align*}
		\int_{{\mathbb{R}^n}} {{x^\alpha }{\partial ^\beta }K\left( x \right){\rm d}x}  &= \int_{{\mathbb{R}^n}} {\Big( {\prod\limits_{j = 1}^n {x_j^{{\alpha _j}}} } \Big)\Big( {\prod\limits_{j = 1}^n {\partial _{{x_j}}^{{\beta _j}}} } \Big)K\left( x \right){\rm d}x} \\
		&= {\left( { - 1} \right)^{{\beta _1} - {\alpha _1}}}\int_{{\mathbb{R}^n}} {\Big( {\prod\limits_{j = 2}^n {x_j^{{\alpha _j}}} } \Big)\Big( {\prod\limits_{j = 2}^n {\partial _{{x_j}}^{{\beta _j}}} } \Big)\partial _{{x_1}}^{{\beta _1} - {\alpha _1}}K\left( x \right){\rm d}x}=0.
	\end{align*}
	{\bf{Case3:}} Now we have $ {\alpha _1} \geq {\beta _1} $, and some $ {\alpha _j} \geq {\beta _j} + 1 $ (otherwise $ \alpha  = \beta  $). Let $ j=1 $ without loss of generality. At this time we first prove a conclusion according to \eqref{3.3}. Since
	\[{\partial ^\alpha }\widehat K\left( 0 \right) = {\left( { - \mathrm{i}} \right)^{\left| \alpha  \right|}}\int_{{\mathbb{R}^n}} {{x^\alpha }K\left( x \right){\rm d}x}  = 0,\;\;\alpha  \ne 0,\]
 it follows that
	\begin{displaymath}
		\int_{{\mathbb{R}^n}} {{x^\alpha }K\left( x \right){\rm d}x}  = 0,\;\;\alpha  \ne 0.
	\end{displaymath}
	Hence, we arrive at
	\[\int_{{\mathbb{R}^n}} {{x^\alpha }{\partial ^\beta }K\left( x \right){\rm d}x}  = {\left( { - 1} \right)^{\sum\limits_{j = 1}^n {\left( {{\beta _j} - {\alpha _j}} \right)} }}\int_{{\mathbb{R}^n}} {\left( {x_1^{{\alpha _1} - {\beta _1}}} \right)\Big( {\prod\limits_{j = 2}^n {x_j^{{\alpha _j} - {\beta _j}}} } \Big)K\left( x \right){\rm d}x}  = 0.\]
	This proves \eqref{3.7}. Next, we will consider a complex translation of \eqref{3.7} and prove that
	\begin{equation}\notag
		\int_{{\mathbb{R}^n}} {{{\left( {u + \mathrm{i}v} \right)}^\alpha }{\partial ^\beta }K\left( {u + \mathrm{i}v} \right){\rm d}u}  = \left\{ \begin{aligned}
			&{\left( { - 1} \right)^{\left| \alpha  \right|}}\alpha !,&\alpha  = \beta,  \hfill \\
			&0,&\alpha  \ne \beta . \hfill \\
		\end{aligned}  \right.
	\end{equation}
	Actually one only needs to pay attention to \eqref{5}, and the proof is completed according to Cauchy's integral formula.
	
	Finally, we only prove that
	\begin{equation}\label{3.10}
		{S_r}{p} = {p},\;\;{p}:{\mathbb{R}^n} \to \mathbb{R}.
	\end{equation}
	In fact, only real polynomials need to be considered
	\[{p} = {x^\alpha } = \prod\limits_{j = 1}^n {x_j^{{\alpha _j}}} \ .\]
	It can be obtained by straight calculation
	\begin{align*}
		{S_r}{p} &= {r^{ - n}}\int_{{\mathbb{R}^n}} {K\left( {\frac{{x - y}}{r}} \right)\prod\limits_{j = 1}^n {y_j^{{\alpha _j}}} {\rm d}y}  = \int_{{\mathbb{R}^n}} {K\left( z \right)\prod\limits_{j = 1}^n {{{\left( {r{z_j} + {x_j}} \right)}^{{\alpha _j}}}} {\rm d}z} \\
		& = \Big( {\prod\limits_{j = 1}^n {x_j^{{\alpha _j}}} } \Big)\int_{{\mathbb{R}^n}} {K\left( z \right){\rm d}z}  + \sum\limits_\gamma  {{\varphi _\gamma }\left( {r,x} \right)\int_{{\mathbb{R}^n}} {{z^\gamma }K\left( z \right){\rm d}z} }  = \prod\limits_{j = 1}^n {x_j^{{\alpha _j}}}  = {p},
	\end{align*}
	where $ {{\varphi _\gamma }\left( {r,x} \right)} $ are coefficients independent of  $ z $.
	
	As to the complex case one only needs to perform complex translation to obtain
	\begin{equation}\notag
		{p}\left( {u;\mathrm{i}v} \right) = {S_r}{p}\left( {u + \mathrm{i}v} \right) = \int_{{\mathbb{R}^n}} {K\left( {\mathrm{i}{r^{ - 1}}v - \eta } \right){p}\left( {u;r\eta } \right){\rm d}\eta }.
	\end{equation}

	The above preparations are classical, see also \cite{salamon,MR2071231}. Next we begin to prove the Jackson type approximation theorem via only modulus of continuity. 	We will make use of \eqref{3.10} in case of the Taylor polynomial
	\[{p_k}\left( {x;y} \right): = {P_{f,k}}\left( {x;y} \right) = \sum\limits_{\left| \alpha  \right| \leq k} {\frac{1}{{\alpha !}}{\partial ^\alpha }f\left( x \right){y^\alpha }} \]
	of $ f $ with $ k\in \mathbb{N} $. Note that
	\[\left| {f\left( {x + y} \right) - {p_k}\left( {x;y} \right)} \right| = \Bigg| {\int_0^1 {k{{\left( {1 - t} \right)}^{k - 1}}\sum\limits_{\left| \alpha  \right| = k} {\frac{1}{{\alpha !}}\left( {{\partial ^\alpha }f\left( {x + ty} \right) - {\partial ^\alpha }f\left( x \right)} \right){y^\alpha }{\rm d}t} } } \Bigg|\]
	for every $ x,y \in {\mathbb{R}^n} $.
	
	Define the following domains to partition $ \mathbb{R}^n $:
	\[{\Omega _1}: = \left\{ {\eta  \in {\mathbb{R}^n}:\left| \eta  \right| < \delta{r^{ - 1}}} \right\},\;\;{\Omega _2}: = \left\{ {\eta  \in {\mathbb{R}^n}:\left| \eta  \right| \geq \delta{r^{ - 1}}} \right\}.\]
	We have to use different estimates in the above two domains, which are abstracted as follows.		If $ 0 < \left| y \right| < \delta $,  we obtain that
	\begin{align}
		\left| {f\left( {x + y} \right) - {p_k}\left( {x;y} \right)} \right| \leq{}& \int_0^1 {k{{\left( {1 - t} \right)}^{k - 1}}\sum\limits_{\left| \alpha  \right| = k} {\frac{1}{{\alpha !}} \cdot {{\left[ {{\partial ^\alpha }f} \right]}_\varpi }\varpi \left( {t\left| y \right|} \right) \cdot \left| {{y^\alpha }} \right|{\rm d}t} } \notag \\
		\leq{}& c\left( {n,k} \right){\left\| f \right\|_\varpi }\int_0^1 {\varpi \left( {t\left| y \right|} \right){\rm d}t}  \cdot \left| {{y^\alpha }} \right|\notag \\
		\leq{}& c\left( {n,k} \right){\left\| f \right\|_\varpi }\varpi \left( {\left| y \right|} \right)\left| {{y^\alpha }} \right|\notag \\
		\label{e1}\leq{}& c\left( {n,k} \right){\left\| f \right\|_\varpi }\varpi \left( {\left| y \right|} \right){\left| y \right|^k}.
	\end{align}
	If $ \left| y \right| \geq \delta $, one easily arrives at
	\begin{align}
		\left| {f\left( {x + y} \right) - {p_k}\left( {x;y} \right)} \right| \leq{}& \int_0^1 {k{{\left( {1 - t} \right)}^{k - 1}}\sum\limits_{\left| \alpha  \right| = k} {\frac{1}{{\alpha !}} \cdot 2{{\left| {{\partial ^\alpha }f} \right|}_{{C^0}}} \cdot \left| {{y^\alpha }} \right|{\rm d}t} } \notag \\
		\leq{}& c\left( {n,k} \right){\left\| f \right\|_\varpi }\left| {{y^\alpha }} \right|\notag \\
		\label{e2}	\leq{}& c\left( {n,k} \right){\left\| f \right\|_\varpi }{\left| y \right|^k}.
	\end{align}
	The H\"{o}lder inequality has been used in \eqref{e1} and \eqref{e2} with $ k \geq 1,{\alpha _i} \geq 1, \mu_i=k/\alpha_i\geq 1 $ without loss of generality:
	\[\left| {{y^\alpha }} \right| = \prod\limits_{i = 1}^n {{{\left| {{y_i}} \right|}^{{\alpha _i}}}}  \leq \sum\limits_{i = 1}^n {\frac{1}{{{\mu _i}}}{{\left| {{y_i}} \right|}^{{\alpha _i}{\mu _i}}}}  \leq \sum\limits_{i = 1}^n {{{\left| {{y_i}} \right|}^k}}  \leq \sum\limits_{i = 1}^n {{{\left| y \right|}^k}}  = n{\left| y \right|^k}.\]
	
	Now let $ x=u+\mathrm{i}v $ with $ u,v \in {\mathbb{R}^n} $ and $ \left| v \right| \leq r $. Fix $ p = n + k + 2 $, and let $ c = c\left( {n,k} \right) > 0 $ be a universal constant, then it follows that
	\begin{align*}
		\left| {{S_r}f\left( {u + \mathrm{i}v} \right) - {p_k}\left( {u;\mathrm{i}v} \right)} \right| \leq{}& \int_{{\mathbb{R}^n}} {K\left( {\mathrm{i}{r^{ - 1}}v - \eta } \right)\left| {f\left( {u + r\eta } \right) - {p_k}\left( {u;r\eta } \right)} \right|{\rm d}\eta } \notag \\
		\leq{}& c\int_{{\mathbb{R}^n}} {\frac{{{e^{{r^{ - 1}}v}}}}{{{{\left( {1 + \left| \eta  \right|} \right)}^p}}}\left| {f\left( {u + r\eta } \right) - {p_k}\left( {u;r\eta } \right)} \right|{\rm d}\eta } \notag \\
		\leq{}& c\int_{{\mathbb{R}^n}} {\frac{1}{{{{\left( {1 + \left| \eta  \right|} \right)}^p}}}\left| {f\left( {u + r\eta } \right) - {p_k}\left( {u;r\eta } \right)} \right|{\rm d}\eta } \notag \\
		={}& c\int_{{\Omega _1}}  +  \int_{{\Omega _2}} {\frac{1}{{{{\left( {1 + \left| \eta  \right|} \right)}^p}}}\left| {f\left( {u + r\eta } \right) - {p_k}\left( {u;r\eta } \right)} \right|{\rm d}\eta } \notag \\
		: ={}& c\left( {{I_1} + {I_2}} \right).
	\end{align*}
	As it can be seen later, $ I_1 $ is the main part while $ I_2 $ is the remainder.
	
	Recall Remark \ref{Remarksemi} and \eqref{Ox}. Hence the following holds due to \eqref{e1}
	\begin{align}\label{I1}
		{I_1} ={}& \int_{{\Omega _1}} {\frac{1}{{{{\left( {1 + \left| \eta  \right|} \right)}^p}}}\left| {f\left( {u + r\eta } \right) - {p_k}\left( {u;r\eta } \right)} \right|{\rm d}\eta }\notag \\
		\leq{}& \int_{\left| \eta  \right| < \delta{r^{ - 1}}} {\frac{1}{{{{\left( {1 + \left| \eta  \right|} \right)}^p}}} \cdot c{{\left\| f \right\|}_\varpi }\varpi \left( {\left| {r\eta } \right|} \right){{\left| {r\eta } \right|}^k}{\rm d}\eta } \notag \\
		\leq{}& \int_{\left| \eta  \right| < \delta{r^{ - 1}}} {\frac{1}{{{{\left( {1 + \left| \eta  \right|} \right)}^p}}} \cdot c{{\left\| f \right\|}_\varpi }\varpi \left( { {r } } \right) \psi(|\eta|) {{\left| {r\eta } \right|}^k}{\rm d}\eta } \notag \\
		\leq{}& c{\left\| f \right\|_\varpi }{r^k}{\varpi(r)}\int_0^{\delta{r^{ - 1}}} {\frac{{{w^{k + n }}}}{{{{\left( {1 + w} \right)}^p}}}{\rm d}w} \notag \\
		\leq{}& c{\left\| f \right\|_\varpi }{r^k}{\varpi(r)}\int_0^{+\infty} {\frac{{{w^{k + n }}}}{{{{\left( {1 + w} \right)}^p}}}{\rm d}w} \notag \\
		\leq{}& c{\left\| f \right\|_\varpi }{r^k}{\varpi(r)}.
	\end{align}
	In view of \eqref{e2}, we have
	\begin{align}\label{I2}
		{I_2} ={}& \int_{{\Omega _2}} {\frac{1}{{{{\left( {1 + \left| \eta  \right|} \right)}^p}}}\left| {f\left( {u + r\eta } \right) - {p_k}\left( {u;r\eta } \right)} \right|{\rm d}\eta } \notag \\
		\leq{}& \int_{\left| \eta  \right| \geq \delta{r^{ - 1}}} {\frac{1}{{{{\left( {1 + \left| \eta  \right|} \right)}^p}}} \cdot c{{\left\| f \right\|}_\varpi }{{\left| {r\eta } \right|}^k}{\rm d}\eta } \notag \\
		\leq{}& c{\left\| f \right\|_\varpi }{r^k}\int_{\delta{r^{ - 1}}}^{ + \infty } {\frac{{{w^{k + n - 1}}}}{{{{\left( {1 + w} \right)}^p}}}{\rm d}w} \notag \\
		\leq{}& c{\left\| f \right\|_\varpi }{r^k}\int_{\delta{r^{ - 1}}}^{ + \infty } {\frac{1}{{{w^{p - k - n + 1}}}}{\rm d}w} \notag \\
		\leq{}& c{\left\| f \right\|_\varpi }{r^{k+2}}.
	\end{align}
	By \eqref{I1} and \eqref{I2}, we finally arrive at
	\begin{equation}\notag
		\left| {{S_r}f\left( {u + \mathrm{i}v} \right) - {p_k}\left( {u;\mathrm{i}v} \right)} \right| \leq c{\left\| f \right\|_\varpi }{r^k} {\varpi(r)}
	\end{equation}
	due to $ \mathop {\overline {\lim } }\limits_{r \to {0^ + }} r/{\varpi }\left( r \right) <  + \infty  $ in Definition \ref{d1}.	This proves Theorem \ref{Theorem1} for the case $ |\alpha| = 0 $. As to $ |\alpha| \ne 0 $, the result follows from the fact that $ {S_r} $ commutes with $ {\partial ^\alpha } $. We therefore finish the proof of Theorem \ref{Theorem1}.
\end{proof}

\subsection{Proof of Corollary \ref{coro1}}\label{proofcoro1}
\begin{proof}
	Only the analysis of case $ \left| \alpha  \right| = 0 $ is given. In view of  Theorem \ref{Theorem1} and \eqref{e1}, we obtain that
	\begin{align}
		\left| {{S_r}f\left( x \right) - f\left( x \right)} \right| \leq{}& \left| {{S_r}f\left( x \right) - {P_{f,k}}\left( {\operatorname{Re} x;\mathrm{i}\operatorname{Im} x} \right)} \right| + \left| {{P_{f,k}}\left( {\operatorname{Re} x;\mathrm{i}\operatorname{Im} x} \right) - f\left( x \right)} \right|\notag \\
		\label{Srf-f}	\leq{}& c_*{\left\| f \right\|_\varpi }{r^k}\varpi(r) ,
	\end{align}
	where the constant $ c_*>0 $ depends on $ n $ and $ k $.
	Further, by \eqref{Srf-f} we have
	\begin{align*}
		\left| {{S_r}f\left( x \right)} \right| \leq{}& \left| {{S_r}f\left( x \right) - f\left( x \right)} \right| + \left| {f\left( x \right)} \right|\notag \\
		\leq{}& c_*{\left\| f \right\|_\varpi }{r^k}\varpi(r) + {\left\| f \right\|_\varpi } \leq {c^ * }{\left\| f \right\|_\varpi },
	\end{align*}
	provided a constant $ c^*>0 $ depending on $ n,k $ and $ \varpi $.	This completes the proof.
\end{proof}

\subsection{Proof of Corollary \ref{coro2}}\label{proofcoco2}
\begin{proof}
	It is easy to verify that
	\begin{align*}
		{S_r}f\left( {x + 1} \right) ={}& \frac{1}{{{r^n}}}\int_{{\mathbb{R}^n}} {K\left( {\frac{{x - \left( {y - 1} \right)}}{r}} \right)f\left( y \right){\rm d}y} \notag \\
		={}& \frac{1}{{{r^n}}}\int_{{\mathbb{R}^n}} {K\left( {\frac{{x - u}}{r}} \right)f\left( {u + 1} \right){\rm d}u} \notag \\
		={}& \frac{1}{{{r^n}}}\int_{{\mathbb{R}^n}} {K\left( {\frac{{x - u}}{r}} \right)f\left( u \right){\rm d}u}  = {S_r}f\left( x \right).
	\end{align*}
	According to Fubini's theorem, we obtain
	\begin{align*}
		\int_{{\mathbb{T}^n}} {{S_r}f\left( x \right)dx}  ={}& \frac{1}{{{r^n}}}\int_{{\mathbb{R}^n}} {\int_{{\mathbb{T}^n}} {K\left( {\frac{{x - y}}{r}} \right)f\left( y \right){\rm d}y} }  \notag \\
		={}& \frac{1}{{{r^n}}}\int_{{\mathbb{R}^n}} {K\left( {\frac{m}{r}} \right)\left( {\int_{{\mathbb{T}^n}} {f\left( {x + m} \right){\rm d}x} } \right) {\rm d}m}  = 0.
	\end{align*}
	
	This completes the proof.
\end{proof}

\subsection{Asymptotic analysis}
Here we provide some useful asymptotic results, all of which can be proved by L'Hopital's rule or by integration by parts,  thus the proof is omitted here.
\begin{lemma}\label{duochongduishu}
	Let $ \varrho \in \mathbb{N}^+ $, $ \lambda>1 $ and some $ M>0 $ sufficiently large be fixed. Then as $ X\to +\infty $,
	\[\int_M^X {\frac{1}{{(\ln z) \cdots {{(\underbrace {\ln  \cdots \ln }_\varrho z)}^\lambda }}}{\rm d}z}  = {\mathcal{O}^\# }\Bigg( {\frac{X}{{(\ln X) \cdots {{(\underbrace {\ln  \cdots \ln }_\varrho X)}^\lambda }}}} \Bigg).\]
\end{lemma}

\begin{lemma}\label{erheyi}
	Let $ 0 <\sigma<1 $, $ \lambda>1 $ and some $ M>0 $ sufficiently large be fixed. Then as $ X\to +\infty $,
	\begin{equation}\label{erheyi1}
		\int_M^X {\frac{1}{{{z^\sigma }{{\left( {\ln z} \right)}^\lambda }}}{\rm d}z}  = {\mathcal{O}^\# }\left( {\frac{{{X^{1 - \sigma }}}}{{{{\left( {\ln X} \right)}^\lambda }}}} \right),
	\end{equation}
	and
	\begin{equation}\label{erheyi2}
		\int_X^{ + \infty } {\frac{1}{{{z^{1 + \sigma }}{{\left( {\ln z} \right)}^\lambda }}}{\rm d}z}  = {\mathcal{O}^\# }\left( {\frac{1}{{{X^\sigma }{{\left( {\ln X} \right)}^\lambda }}}} \right).
	\end{equation}
\end{lemma}

\subsection{KAM theorem for quantitative estimates}\label{Appsalamon}
Here we give a KAM theorem for quantitative estimates, which is used in  Theorem \ref{theorem1} in this paper. See Theorem 1 in Salamon's paper \cite{salamon} for details.
\begin{theorem}\label{appendix}
	Let $ n \geq 2, \tau > n - 1,  0 < \theta < 1$, and $ M \geq 1 $ be given. Then there are positive constants $ \delta_* $ and $ c $ such that $ c\delta_*\leq1/2 $ and the following holds for every $ 0 < r^* \leq 1 $ and every $ \omega\in\mathbb{R}^n $ that satisfies \eqref{dio} (for $ \tau>n-1 $).
	
	Suppose that $ H(x, y) $ is a real analytic Hamiltonian function defined in the strip
	$ \left| {\operatorname{Im} x} \right| \leq {r^ * },\left| y \right| \leq {r^ * } $, which is of period $ 1 $ in the variables $ {x_1}, \ldots ,{x_n} $ and satisfies
	\begin{align*}
		\left| {H\left( {x,0} \right) - \int_{{\mathbb{T}^n}} {H\left( {\xi ,0} \right)d\xi } } \right| &\leq {\delta ^ * }{r^ * }^{2\tau  + 2},\notag\\
		\left| {{H_y}\left( {x,0} \right) - \omega } \right| &\leq {\delta ^ * }{r^ * }^{\tau  + 1},\notag\\
		\left| {{H_{yy}}\left( {x,y} \right) - Q\left( {x,y} \right)} \right| &\leq \frac{{c{\delta ^ * }}}{{2M}},\notag
	\end{align*}
	for $ \left| {\operatorname{Im} x} \right| \leq r^*,\left| y \right| \leq r^* $, where $ 0 < {\delta ^ * } \leq {\delta _ * } $, and $ Q\left( {x,y} \right) \in {\mathbb{C}^{n \times n}} $ is a symmetric
	(not necessarily analytic) matrix valued function in the strip $ \left| {\operatorname{Im} x} \right| \leq r,\left| y \right| \leq r $
	and satisfies in this domain
	\[\left| {Q\left( z \right)} \right| \leq M,\;\;\left| {{{\left( {\int_{{\mathbb{T}^n}} {Q\left( {x,0} \right){\rm d}x} } \right)}^{ - 1}}} \right| \leq M.\]
	Then there exists a real analytic symplectic transformation $ z = \phi \left( \zeta  \right) $ of the
	form
	\[z = \left( {x,y} \right),\;\;\zeta  = \left( {\xi ,\eta } \right),\;\;x = u\left( \xi  \right),\;\;y = v\left( \xi  \right) + u_\xi ^{\top}{\left( \xi  \right)^{ - 1}}\eta \]
	mapping the strip $ \left| {\operatorname{Im} \xi } \right| \leq \theta r^*,\left| \eta  \right| \leq \theta r^* $ into $ \left| {\operatorname{Im} x} \right| \leq r^*,\left| y \right| \leq r^* $, such that $ u\left( \xi  \right) - \xi  $ and $ v\left( \xi  \right) $ are of period $ 1 $ in all variables and the Hamiltonian function $ K: = H \circ \phi  $ satisfies
	\[{K_\xi }\left( {\xi ,0} \right) = 0,\;\;{K_\eta }\left( {\xi ,0} \right) = \omega .\]
	Moreover, $ \phi $ and $ K $ satisfy the estimates
	\begin{align*}
		&\left| {\phi \left( \zeta  \right) - \zeta } \right| \leq c{\delta ^ * }\left( {1 - \theta } \right){r^ * },\;\;\left| {{\phi _\zeta }\left( \zeta  \right) - \mathbb{I}} \right| \leq c{\delta ^ * },\\
		&\left| {{K_{\eta \eta }}\left( \zeta  \right) - Q\left( \zeta  \right)} \right| \leq \frac{{c{\delta ^ * }}}{M},\\
		&\left| {v \circ {u^{ - 1}}\left( x \right)} \right| \leq c{\delta ^ * }{r^ * }^{\tau  + 1},
	\end{align*}
	for $ \left| {\operatorname{Im} \xi } \right| \leq \theta r^*,\left| \eta  \right| \leq \theta r^* $, and $ \left| {\operatorname{Im} x} \right| \leq \theta r^* $.
\end{theorem}

 \section*{Acknowledgements} 
This work was supported in part by National Basic Research Program of China (Grant No. 2013CB834100), National Natural Science Foundation of China (Grant No. 12071175, Grant No. 11171132,  Grant No. 11571065), Project of Science and Technology Development of Jilin Province (Grant No. 2017C028-1, Grant No. 20190201302JC), and Natural Science Foundation of Jilin Province (Grant No. 20200201253JC).


\begin{thebibliography}{99}
	
\bibitem{Chaotic}
{\sc J.~Albrecht}, {\em On the existence of invariant tori in nearly-integrable Hamiltonian systems with finitely differentiable perturbations}, Regul. Chaotic Dyn., 12 (2007), pp.~281--320, \url{https://doi.org/10.1134/S1560354707030033}.


\bibitem{R-9}
{\sc V.~I. Arnold}, {\em Small denominators. {I}. {M}apping the circle onto
	itself}, Izv. Akad. Nauk SSSR Ser. Mat., 25 (1961), pp.~21--86.

\bibitem{R-10}
{\sc V.~I. Arnold}, {\em Proof of a theorem of {A}. {N}. {K}olmogorov on the
	preservation of conditionally periodic motions under a small perturbation of
	the {H}amiltonian}, Uspehi Mat. Nauk, 18 (1963), pp.~13--40.

\bibitem{R-11}
{\sc V.~I. Arnold}, {\em Small denominators and problems of stability of motion
	in classical and celestial mechanics}, Uspehi Mat. Nauk, 18 (1963),
pp.~91--192.


\bibitem{MR2269239}
{\sc V. I. Arnold, V. V. Kozlov, A. I. Neishtadt}, {\em  Mathematical aspects of classical and celestial mechanics}, [Dynamical systems. III]. Translated from the Russian original by E. Khukhro. Third edition. Encyclopaedia of Mathematical Sciences, 3. Springer-Verlag, Berlin, 2006. pp.~xiv+518.


\bibitem{Bounemoura}
{\sc A.~Bounemoura}, {\em Positive measure of {KAM} tori for finitely
	differentiable {H}amiltonians}, J. \'{E}c. polytech. Math., 7 (2020),
pp.~1113--1132, \url{https://doi.org/10.5802/jep.137}.




\bibitem{MR1316975}
{\sc J. Bourgain}, {\em Construction of quasi-periodic solutions for Hamiltonian perturbations of linear equations and applications to nonlinear PDE}, Int. Math. Res. Not. 11 (1994), pp.~475--497, \url{https://doi.org/10.1155/S1073792894000516}.


\bibitem{MR1345016}
{\sc J. Bourgain}, {\em Construction of periodic solutions of nonlinear wave equations in higher dimension},  Geom. Funct. Anal. 5 (1995), pp.~629--639, \url{https://doi.org/10.1007/BF01902055}.

\bibitem{MR3061774}
{\sc C.-Q. Cheng, L. Wang}, 
{\em Destruction of Lagrangian torus for positive definite Hamiltonian systems}, 
Geom. Funct. Anal. 23 (2013),  pp.~848--866,
\url{https://doi.org/10.1007/s00039-013-0213-z}.


\bibitem{MR2071231}
{\sc L. Chierchia}, {\em KAM lectures}, Dynamical systems. Part I, pp.~1--55, Pubbl. Cent. Ric. Mat. Ennio Giorgi, Scuola Norm. Sup., Pisa, 2003.





\bibitem{R-18}
{\sc J. Du, Y. Li, H. Zhang},  {\em Kolmogorov's theorem for degenerate Hamiltonian systems with continuous parameters}, \url{https://doi.org/10.48550/arXiv.2206.05461}.


\bibitem{MR1001032}
{\sc L. H. Eliasson}, {\em Perturbations of stable invariant tori for Hamiltonian systems}, 
Ann. Scuola Norm. Sup. Pisa Cl. Sci. (4) 15 (1988),  pp.~115--147 (1989), \url{http://www.numdam.org/item?id=ASNSP_1988_4_15_1_115_0}.



\bibitem{MR3357183}
{\sc L. H. Eliasson, B. Fayad, R. Krikorian}, {\em Around the stability of KAM tori}, Duke Math. J. 164 (2015),  pp.~1733--1775, \url{https://doi.org/10.1215/00127094-3120060}.



\bibitem{MR0656198}
{\sc R. Hamilton}, {\em The inverse function theorem of Nash and Moser},  Bull. Amer. Math. Soc. (N.S.) 7 (1982),  pp.~65--222, \url{https://doi.org/10.1090/S0273-0979-1982-15004-2}.




\bibitem{Herman3}
{\sc M.-R.~Herman}, {\em Sur la conjugaison diff\'{e}rentiable des diff\'{e}omorphismes du
	cercle \`a des rotations}, Inst. Hautes \'{E}tudes Sci. Publ. Math., 49 (1979),
pp.~2--233.


\bibitem{M1}
{\sc M.-R.~Herman}, {\em Sur les courbes invariantes par les
	diff\'{e}omorphismes de l'anneau. {V}ol. 1}, Ast\'{e}risque, 103 (1983),
pp.~i+221.

\bibitem{M2}
{\sc M.-R.~Herman}, {\em Sur les courbes invariantes par les
	diff\'{e}omorphismes de l'anneau. {V}ol. 2}, Ast\'{e}risque, 103 (1986),
p.~248.


\bibitem{MR0802486}
{\sc L. H\"{o}rmander}, {\em On the Nash-Moser implicit function theorem}, Ann. Acad. Sci. Fenn. Ser. A I Math. 10 (1985), pp.~255--259, \url{https://doi.org/10.5186/aasfm.1985.1028}.


\bibitem{R-2}
{\sc H.~Jacobowitz}, {\em Implicit function theorems and isometric embeddings},
Ann. of Math. (2), 95 (1972), pp.~191--225,
\url{https://doi.org/10.2307/1970796}.

\bibitem{MR1036903}
{\sc Y. Katznelson, D. Ornstein},  {\em The absolute continuity of the conjugation of certain diffeomorphisms of the circle},  Ergodic Theory Dynam. Systems 9 (1989),  pp.~681--690,
\url{https://doi.org/10.1017/S0143385700005289}

\bibitem{MR3269186}
{\sc B. Khesin, S. Kuksin, D.  Peralta-Salas}, {\em KAM theory and the 3D Euler equation},
Adv. Math. 267 (2014), pp.~498--522,
\href{https://doi.org/10.1016/j.aim.2014.09.009}{https://doi.org/10.1016/j.aim.2014.09.009}.


\bibitem{MR0068687}
{\sc A. N. Kolmogorov}, {\em On conservation of conditionally periodic motions for a small change in Hamilton's function},  Dokl. Akad. Nauk SSSR (N.S.) 98, (1954), pp.~527--530.

\bibitem{Koudjinan}
{\sc C.~E.~Koudjinan}, {\em A {KAM} theorem for finitely differentiable
	{H}amiltonian systems}, J. Differential Equations, 269 (2020),
pp.~4720--4750, \url{https://doi.org/10.1016/j.jde.2020.03.044}.

\bibitem{MR0911772}
{\sc S. Kuksin}, {\em Hamiltonian perturbations of infinite-dimensional linear systems with imaginary spectrum}, Funktsional. Anal. i Prilozhen. 21 (1987), pp.~22--37.

\bibitem{R-55}
{\sc J.~N.~Mather}, {\em Nonexistence of invariant circles}, Ergodic Theory
Dynam. Systems, 4 (1984), pp.~301--309,
\url{https://doi.org/10.1017/S0143385700002455}.

\bibitem{MR0147741}
{\sc J.~Moser}, {\em On invariant curves of area-preserving mappings of an annulus}, Nachr. Akad. Wiss. G\"{o}ttingen Math.-Phys. Kl. II 1962 (1962), pp.~1--20.

\bibitem{R-12}
{\sc J.~Moser}, {\em A rapidly convergent iteration method and non-linear	partial differential equations. {I}}, Ann. Scuola Norm. Sup. Pisa Cl. Sci. (3), 20 (1966), pp.~265--315.


\bibitem{R-13}
{\sc J.~Moser}, {\em A rapidly convergent iteration method and non-linear
	differential equations. {II}}, Ann. Scuola Norm. Sup. Pisa Cl. Sci. (3), 20
(1966), pp.~499--535.



\bibitem{J-3}
{\sc J.~Moser}, {\em On the construction of almost periodic solutions for ordinary differential equations}, Proc. {I}nternat. {C}onf. on {F}unctional	{A}nalysis and {R}elated {T}opics ({T}okyo, 1969),  (1970), pp.~60--67.

\bibitem{MR75639}
{\sc J. Nash}, {\em The imbedding problem for Riemannian manifolds}, 
Ann. of Math. (2) 63 (1956), pp.20--63, \url{https://doi.org/10.2307/1969989}.












\bibitem{Po1}
{\sc J.~P\"{o}schel}, {\em \"{U}ber invariante {T}ori in differenzierbaren
	{H}amiltonschen {S}ystemen}, Beitr\"{a}ge zur Differentialgeometrie
[Contributions to Differential Geometry], 3 (1980), p.~103.

\bibitem{Po2}
{\sc J.~P\"{o}schel}, {\em Integrability of {H}amiltonian systems on {C}antor
	sets}, Comm. Pure Appl. Math., 35 (1982), pp.~653--696,
\url{https://doi.org/10.1002/cpa.3160350504}.

\bibitem{MR1022821}
{\sc J.~P\"{o}schel}, {\em On elliptic lower-dimensional tori in Hamiltonian systems},  Math. Z. 202 (1989),  pp.~559--608, \url{https://doi.org/10.1007/BF01221590}.





\bibitem{Po3}
{\sc J.~P\"{o}schel}, {\em {KAM} below $ \mathrm{C}^n $}, \url{https://arxiv.org/abs/2104.01866}.






\bibitem{salamon}
{\sc D.~A.~Salamon}, {\em The {K}olmogorov-{A}rnold-{M}oser theorem}, Math.
Phys. Electron. J., 10 (2004), pp.~Paper 3, 37.




\bibitem{Stein}
{\sc E.~M.~Stein, G.~Weiss}, {\em Introduction to Fourier analysis on Euclidean spaces}, Princeton Mathematical Series, No. 32. Princeton University Press, Princeton, N.J., (1971), pp.~x+297.


\bibitem{F-1}
{\sc F.~Takens}, {\em A {$C^{1}$} counterexample to {M}oser's twist theorem},
Nederl. Akad. Wetensch. Proc. Ser. A {\bf 74}=Indag. Math., 33 (1971),
pp.~378--386.


\bibitem{TD} {\sc Z. Tong, J. Du, Y. Li}, {\em KAM theorem on modulus of continuity about parameter}, Sci. China Math (2023),  \url{https://doi.org/10.1007/s11425-022-2102-5}.



\bibitem{MR4385768}
{\sc L. Wang}, {\em Quantitative destruction of invariant circles},
Discrete Contin. Dyn. Syst. 42 (2022), pp.~1569--1583,
\url{https://doi.org/10.3934/dcds.2021164}.

\bibitem{MR1040892}
{\sc C. E. Wayne}, {\em Periodic and quasi-periodic solutions of nonlinear wave equations via KAM theory}, Comm. Math. Phys. 127 (1990), pp.~479--528, \url{http://projecteuclid.org/euclid.cmp/1104180217}.



\bibitem{R-7}
{\sc E.~Zehnder}, {\em Generalized implicit function theorems with applications
	to some small divisor problems. {I}}, Comm. Pure Appl. Math., 28 (1975),
pp.~91--140, \url{https://doi.org/10.1002/cpa.3160280104}.

\bibitem{R-8}
{\sc E.~Zehnder}, {\em Generalized implicit function theorems with applications
	to some small divisor problems. {II}}, Comm. Pure Appl. Math., 29 (1976),
pp.~49--111, \url{https://doi.org/10.1002/cpa.3160290104}.
\end{thebibliography}
\end{document}